\documentclass[12pt]{elsarticle}
\usepackage{amsmath,amsthm}
\usepackage{amssymb,latexsym}
\usepackage{enumerate}
\usepackage[all]{xy}

\newtheorem{thm}{Theorem}
\newtheorem{prop}[thm]{Proposition}
\newtheorem{lem}[thm]{Lemma}
\newtheorem{cor}[thm]{Corollary}
\theoremstyle{definition}
\newtheorem{defin}[thm]{Definition}
\newtheorem{rem}[thm]{Remark}

\journal{Journal\ldots}

\newcommand{\mk}{\mathfrak}
\newcommand{\ca}{\mathcal}

\newcommand{\wt}{\widetilde}
\newcommand{\ov}{\overline}

\DeclareMathOperator{\Ext}{Ext}
\DeclareMathOperator{\rank}{rank}
\DeclareMathOperator{\Hom}{Hom}
\DeclareMathOperator{\Norm}{N}

\begin{document}

\def\Q{\mathbb{Q}}
\def\Z{\mathbb{Z}}
\def\C{\mathbb{C}}
\def\Gal{\operatorname{Gal}}
\def\res{\operatorname{res}}
\def\Fr{\operatorname{Frob}}
\def\co{\operatorname{cond}}
\def\id{\operatorname{id}}
\def\N{\operatorname{N}}
\def\mm{\mathcal{M}}
\def\M{m}
\def\fq{\mathfrak{q}}
\def\fl{\mathfrak{l}}
\def\fQ{\mathfrak{Q}}
\def\HH{B}
\def\Hm#1#2{\operatorname{Hom}_{{#1}}(#2,{#1})}
\def\lanran#1{\langle#1\rangle }
\def\lanrang#1{\langle#1\rangle_\Gamma }

\def\smj{(1-\sigma)}
\def\smjb{{1-\sigma}}
\def\ann#1{\operatorname{Ann}_{\Z[\Gamma]}(#1)}
\def\en{\pi}
\def\cl{{\operatorname{Cl}}}
\def\Cl{{\operatorname{Cl}(L)}}
\def\im{\operatorname{im}}
\def\hom#1{\operatorname{Hom}(#1,\langle\zeta_\M\rangle)}
\def\Oqqs{(\mathcal{O}_{L[\fq]}/\tilde \fq)^\times}

\def\exponent{k}
\def\ram{t}
\def\dec{n}
\def\GrEllNum{P}
\def\GrEllUn{\mathcal{C}}
\def\zobr{z}

\def\sst{\scriptstyle}
\newcommand\mycom[2]{\genfrac{}{}{0pt}{}{#1}{#2}}
\def\HomR#1{\operatorname{Hom}_R\bigl(#1,R\bigr)}

\title{Annihilators of the ideal class group\\of a cyclic extension of an imaginary quadratic field}

\author[ul]{Hugo Chapdelaine}
\ead{hugo.chapdelaine@mat.ulaval.ca}

\author[mu]{Radan Ku\v cera\fnref{fn1}}
\ead{kucera@math.muni.cz}

\fntext[fn1]{The second author was supported under Project 
15-15785S of the Czech Science Foundation.}

\address[ul]{Faculty of Science and Engineering,
Laval University,
Qu\'ebec G1V 0A6, 
Canada}

\address[mu]{Faculty of Science, 
Masaryk university,
611 37 Brno, Czech Republic
}

\begin{abstract}
The aim of this paper is to study the group of elliptic units of a
cyclic extension $L$ of an imaginary quadratic field $K$ such that the degree $[L:K]$ is a power of an odd prime $p$. 
We construct an explicit root of the usual top generator of this group and 
we use it to obtain an annihilation result of the $p$-Sylow subgroup of the ideal class group of $L$. \par
\end{abstract}


\maketitle


\section*{Introduction}

This work was motivated by the series of papers \cite{AnnI}, \cite{AnnII} and \cite{AnnIII}, which studied annihilators of 
the $p$-Sylow subgroup of the ideal class group of a cyclic abelian field $L$ over $\Q$, whose degree is a power of an odd prime $p$; these 
annihilators were obtained by means of circular units. The goal of this paper is to study annihilators of 
the $p$-Sylow subgroup of the ideal class group of a field $L$ which is a cyclic extension over $K$, 
where $K$ is an imaginary quadratic field whose class number 
$h=h_K$ is not divisible by $p$. In this new setting, the former role played by the circular units is now
being played by the so-called {\it elliptic units}. Similarly to the previous series of papers, certain annihilators of the ideal class group
of $L$ are obtained by means of elliptic units above $K$. Recall that, in essence, an elliptic unit above $K$ is a unit which lives in an abelian extension of $K$ and 
which is obtained by evaluating a certain modular unit (i.e. a modular function whose divisor is supported at the cusps) at an element $\tau\in K\cap\mathfrak{h}$, where
$\mathfrak{h}$ corresponds to the Poincar\'e upper half-plane. Depending on what applications one has in mind, different choices
of modular units have been considered in the literature. For the present paper, we use a slight modification of the group of 
elliptic units introduced by Oukhaba in \cite{H}; the only difference being that we do not raise the generators 
of the group of elliptic units considered in \cite{H} to the $h$-th power. The index of our group 
of elliptic units $\GrEllUn_L$ in the group $\mathcal{O}_{L}^\times$ of all units of $L$ is given in Lemma~\ref{EliptickeJednotky}. 
Then, starting from the group $\GrEllUn_L$, we proceed to extract certain roots (where the root exponents are group ring elements) 
of the generators of $\GrEllUn_L$ which again lie
in $L$. These roots of elliptic units allow us to define an {\it enlarged 
group of elliptic units} $\overline\GrEllUn_L$, whose index in $\mathcal{O}_{L}^\times$ is given in Theorem~\ref{Indexy}.
This enlarged group $\overline\GrEllUn_L$ form an important ingredient of the main result of this paper: any annihilator of the $p$-Sylow 
part of the quotient $\mathcal{O}_{L}^\times/\overline\GrEllUn_L$ must annihilate a certain (very explicit) subgroup of the $p$-Sylow part
of the ideal class group of $L$, see Theorem~\ref{Theorem4}.

We would like to emphasize that many of the techniques used in this paper borrow heavily from the ones introduced in \cite{AnnI} and \cite{AnnIII}. 
In order to keep the paper within a reasonable size, we faced the problem of choosing what proofs to present in full details and what proof to only sketch (or omit). 
For each of the proofs,  we have decided to distinguish whether the needed modifications are straightforward or not. Of course,
such choices are subjective but we hope that our chosen style clarifies the overall presentation, and, at the same time, has
the effect of highlighting the new ideas. For example, in the construction of nontrivial roots of elliptic units given in Sections \ref{Odmocnina} and \ref{Semi},
we decided to give all the details, whereas the necessary modifications of Theorem~\ref{Theorem17} in the style of Rubin are left to the reader.

\section{Notation and preliminaries}\label{bit}
Let $K$ be an imaginary quadratic number field, $H=H_K$ be the Hilbert class field of $K$, $h=h_K=[H:K]$ be the class number of $K$, and let $L/K$ be a 
cyclic Galois extension of degree $p^\exponent$ where $p$ is an odd prime and $k$ is a positive integer. We let
$\Gamma=\Gal(L/K)=\langle\sigma\rangle$ where $\sigma$ is a fixed generator.
We suppose that $p\nmid h$ and that there are exactly $s\ge2$ ramified primes in $L/K$. It follows from the first
assumption that $L\cap H=K$. 
Let $\wp_1,\dots,\wp_s$ be all the (pairwise distinct) prime ideals of $K$ which ramify in $L/K$. 
For each $j\in I=\{1,\dots,s\}$ we choose a generator $\pi_j\in\mathcal{O}_K$ of the principal ideal $\wp_j^h$ and we let $q_j\in\Z$ be the only rational prime number in $\wp_j$.  
We suppose that $p$ is unramified 
in $L/\Q$ and that each $q_j$ is unramified
in $K/\Q$. In particular, this implies that $p\nmid|\mu_L|$, and that $p\neq q_j$ for all $j\in J$. 
Here $\mu_F$ denotes the group of roots of unity of a field $F$.

For each $j\in I$ let us fix an arbitrarily chosen prime ideal $\mathfrak{P}_j$ of $L$ above $\wp_j$. 
Let $\ram_j$ be the ramification index of $\mathfrak{P}_j$ over $\wp_j$ and let $\dec_j$ be the index of the decomposition group of $\mathfrak{P}_j$ in $\Gamma$.
It follows that $\ram_j\dec_j\mid p^\exponent$ and that $\{\mathfrak{P}_j^{\sigma^{i}}\}_{i=0}^{\dec_j-1}$ is the full set 
of distinct prime ideals of $L$ above $\wp_j$. In particular, we have the following decomposition: 
$$
\wp_j\mathcal{O}_{L}=\prod_{i=0}^{\dec_j-1}\mathfrak{P}_j^{\ram_j\sigma^i}.
$$
We consider the completion $\Q_{q_j}\subseteq K_{\wp_j}\subseteq L_{\mathfrak{P}_j}$ of $\Q\subseteq K\subseteq L$. Since the extension of local fields
$L_{\mathfrak{P}_j}/K_{\wp_j}$ has a ramification index equal to $t_j$, it follows from local 
class field theory that the group $\mathcal{O}_{K_{\wp_j}}^\times$ of units of $\mathcal{O}_{K_{\wp_j}}$ has a closed subgroup of index $\ram_j$, 
namely the subgroup $\N_{L_{\mk{P}_j}/K_{\wp_j}}(\ca{O}_{L_{\mk{P}_j}}^{\times})$. It is well-known that $\mathcal{O}_{K_{\wp_j}}^\times$ is 
the direct product of the group of principal units $\ca{U}_j=\{\epsilon\in\mathcal{O}_{K_{\wp_j}}^\times;\,\epsilon\equiv1\pmod{\wp_j}\}$ 
and of the subgroup of roots of unity of orders coprime to $q_j$, which is a finite cyclic group isomorphic to 
$(\mathcal{O}_{K_{\wp_j}}/\wp_j\mathcal{O}_{K_{\wp_j}})^\times\cong(\mathcal{O}_{K}/\wp_j)^\times$, whose order is $|\mathcal{O}_{K}/\wp_j|-1=\N_{K/\Q}(\wp_j)-1$.
Moreover, it is well-known that if the index of a closed subgroup of $\mathcal{U}_j$ is finite, then this index is a power of $q_j$, 
and so it is coprime to $t_j$ (a power of $p$).
Therefore, we must have $\N_{K/\Q}(\wp_j)\equiv1\pmod {\ram_j}$. 

Since $\wp_1,\dots,\wp_s$ are all the prime ideals which ramify in $L/K$ and there is no real embedding of $K$ we see that the 
conductor of $L/K$ is $\prod_{j\in I}\wp_j^{a_j}$ for some positive integers $a_j\geq 1$. 
Since tamely ramified extensions have square-free conductors (see for example \cite[II.5.2.2(ii) on page 151]{Gras}), we must have $a_j=1$ for all $j\in I$.

\section{The distinguished subfields $F_j$'s}

For each non-zero ideal $\mk{m}\subseteq\ca{O}_K$, let us denote by $K(\mk{m})$ the {\it ray class field of $K$ of modulus $\mk{m}$}.
For any subset $\emptyset\ne J\subseteq I=\{1,\dots,s\}$, we also let $\mathfrak{m}_J=\prod_{j\in J}\wp_j$. 
In the previous section we showed that $L\subseteq K(\mk{m}_I)$. In fact, more is true.
A simple exercise in class field theory shows that the index $[K(\mk{m}_I):\prod_{j\in I} K(\wp_j)]$ divides a power of
$|\mu_K|$, where the product is meant for the compositum of the fields $K(\wp_j)$'s. Since $p\nmid|\mu_K|$, it follows 
that $L\subseteq \prod_{j\in I} K(\wp_j)$.

We would like now to introduce, for each index $j\in I$, a distinguished subfield $F_j\subseteq K(\wp_j)$. 
The following elementary lemma will be used in the definition of $F_j$ and also in Definition \ref{neat4}.

\begin{lem}\label{cyc}
Let $T$ be an abelian group (written additively and not necessarily finite) and let $n$ be a positive integer. If $T/nT\cong\Z/n\Z$ then $T$ admits a unique
subgroup of index $n$, namely $nT$. Let $(T,S,n)$ be a triple such that $T$ is an abelian group, $S\leq T$ is a subgroup of finite index $[T:S]$ and
$n$ is a positive integer. Assume that $\gcd(n,[T:S])=1$. Then the natural map $\pi:S/nS\rightarrow T/nT$ is an isomorphism. 
\end{lem}
\begin{proof}
The elementary proof is left to the reader.  
\end{proof}

From class field theory we have a canonical isomorphism $\Gal(K(\wp_j)/H)\cong (\mathcal{O}_{K}/\wp_j)^\times/\operatorname{im}\mu_K$,
which is a cyclic group of order divisible by $t_j$. Since $p\nmid h$, we may apply Lemma \ref{cyc} to the triple $(\Gal(K(\wp_j)/K),\Gal(K(\wp_j)/H),\ram_j)$ and define
$F_j$ as the unique subfield of $K(\wp_j)$ such that $[F_j:K]=\ram_j$.
One may check that the extension $F_j/K$ satisfies
the following properties: $F_j\cap H=K$ and $F_j/K$ is unramified outside of $\wp_j$ and totally ramified at $\wp_j$.

For any $\emptyset\ne J\subseteq I=\{1,\dots,s\}$, it is convenient to introduce the shorthand notation $K_J=K(\mathfrak{m}_J)$  
and $F_J=\prod_{j\in J}F_j\subseteq K_J$. Note that the conductor of $F_J$ over $K$ is $\mathfrak{m}_J$. 
It follows from the definition of $F_{I}$ that
$\Gal(F_I/F_{I-\{j\}})$ is the inertia subgroup of a prime of $F_I$ above $\wp_j$ (note that $I-\{j\}\neq\emptyset$ since $|I|\geq 2$). In particular, 
for each $j\in J$, $|\Gal(F_I/F_{I-\{j\}})|=t_j$. The next lemma gives the main properties of the Galois extension $F_I/K$. 

\begin{prop}\label{LemmaPrimySoucin}
For each $j\in I$, we have $F_jK_{I-\{j\}}=LK_{I-\{j\}}$. The Galois group
\begin{equation}\label{PrimySoucin}
G=\Gal(F_I/K)=\prod_{j\in I}\Gal(F_I/F_{I-\{j\}})
\end{equation}
is the direct product of its inertia subgroups.
Moreover $L\subseteq F_I$. 
\end{prop}
\begin{proof}
Recall that the conductor of $L$ over $K$ is $\mathfrak{m}_I$, and hence $L\subseteq K_I$. 
For any $j\in I$,
the inertia group of a prime of $K_I$ above $\wp_j$ is $\Gal(K_I/K_{I-\{j\}})$ and so, the inertia group of
a prime of $L$ above $\wp_j$ is $\Gal(L/L\cap K_{I-\{j\}})$ 
(the restriction of $\Gal(K_I/K_{I-\{j\}})$ to $L$).
Hence $\Gal(LK_{I-\{j\}}/K_{I-\{j\}})\cong\Gal(L/L\cap K_{I-\{j\}})$ is of order $\ram_j$. 
An easy ramification argument shows that 
\begin{align}\label{sam}
F_j\cap K_{I-\{j\}}=K.
\end{align}
Indeed, $F_j\cap K_{I-\{j\}}$ is an unramified abelian 
extension of $K$, so that $F_j\cap K_{I-\{j\}}\subseteq H\cap F_j=K$, where the last equality follows from the fact that 
$p\nmid h$. Therefore, $\Gal(F_jK_{I-\{j\}}/K_{I-\{j\}})\cong\Gal(F_j/K)$ is also of order $\ram_j$. 
We thus have proved that the two subgroups $\Gal(K_I/F_jK_{I-\{j\}})$ and $\Gal(K_I/LK_{I-\{j\}})$ have the same index inside
$\Gal(K_I/K_{I-\{j\}})$. Since $K_{I}/K_{I-\{j\}}$ is totally tamely ramified at each prime of $K_I$ above $\wp_j$, it follows that $\Gal(K_I/K_{I-\{j\}})$
is cyclic which forces the group equality
\begin{equation}\label{Inkluze}
\Gal(K_I/F_jK_{I-\{j\}})=\Gal(K_I/LK_{I-\{j\}})\subseteq\Gal(K_I/K_{I-\{j\}}).
\end{equation}
In particular, it follows from \eqref{Inkluze} that $F_jK_{I-\{j\}}=LK_{I-\{j\}}$ which proves 
the first claim. Let us now show \eqref{PrimySoucin}. An argument similar to the proof of \eqref{sam}
implies that $\bigcap_{j\in I} F_{I-\{j\}}=K$, and thus $G$ is generated by 
$\bigcup_{j\in I}\Gal(F_I/F_{I-\{j\}})$. Also, since $F_{I-\{j\}}F_j=F_{I}$, we have $\Gal(F_I/F_{I-\{j\}})\cap\Gal(F_I/F_j)=\{\id\}$
and therefore $G$ is the direct product of the groups $\Gal(F_I/F_{I-\{j\}})$'s which gives \eqref{PrimySoucin}.

It remains to show that $L\subseteq F_I$. Set $M=\bigcap_{j\in I}F_jK_{I-\{j\}}$. Note that $L\subseteq M$ (by the first part of Proposition
\ref{LemmaPrimySoucin}) and that $F_IH\subseteq M$. 

We claim that $F_IH=M$, in particular, this will imply
that $L\subseteq F_IH$. Let us prove it. The inertia group of each prime of $M$ above $\wp_j$  
is of order at most $\ram_j$ since the ramification index of $\wp_j$ in $F_jK_{I-\{j\}}/K$ is equal to $\ram_j$. 
On the one hand, since the maximal unramified subextension of $M/K$ is $H/K$, it follows that (i) $[M:H]\le\prod_{j\in I}\ram_j$.
On the other hand, since $\Gal(F_I/K)$ is a $p$-group and $p\nmid h=[H:K]$, we have $H\cap F_I=K$, so that 
$\Gal(F_IH/H)\cong\Gal(F_I/K)$, and thus from (1) we deduce that (ii) 
$[F_IH:H]=[F_I:K]=\prod_{j\in I}\ram_j$. Combining (i), (ii) with the inclusion $F_IH\subseteq M$ we obtain that $F_IH=M$.

The paragraph above just proved that  $L\subseteq F_IH=M$.
Since $p\nmid h=[F_IH:F_I]$ and $\Gal(F_I/K)$ is a $p$-group,
it follows that $\Gal(F_IH/F_I)$ is the smallest subgroup of the abelian group $\Gal(F_IH/K)$ whose index is a power of $p$, which implies that $L\subseteq F_I$.
This concludes the proof. 
\end{proof}

\begin{cor}\label{neat}
(i) For each index $j\in I$, the inertia subgroup of a prime of $L$ above 
$\wp_j$ is $\Gal(L/L\cap F_{I-\{j\}})=\langle\sigma^{p^k/t_j}\rangle$; 
moreover $F_{I-\{j\}}L=F_I$ and $[L\cap F_{I-\{j\}}:K]=\frac{p^k}{t_j}$.

(ii) $F_I/L$ is an unramified abelian extension.

(iii) There exists at least one index $j_0\in I$ such that $t_{j_0}=p^k$ so that 
the abelian Galois group $G=\Gal(F_I/K)$ has exponent $p^k$.
\end{cor}

\begin{proof} Recall that $\Gal(F_I/F_{I-\{j\}})$ is the inertia subgroup of a prime of $F_I$ above $\wp_j$. We have $L\subseteq F_I$ by Proposition \ref{LemmaPrimySoucin}, and so 
$\Gal(L/L\cap F_{I-\{j\}})$ is the inertia subgroup of a prime of $L$ above $\wp_j$. Since both of
these inertia subgroups have the same order 
$t_j$, and $\langle\sigma^{p^k/t_j}\rangle$ is the only subgroup of $\Gamma$ of order $t_j$,
we get (i) and we see that $F_I/L$ is unramified at each prime of $L$ above $\wp_j$. 
But $F_I/L$ can be ramified only at primes above $\wp_1$, \dots, $\wp_s$ 
because the conductor of $F_I$ over $K$ is $\mathfrak{m}_I$, and (ii) follows. 
By \eqref{PrimySoucin}, the exponent of $G$ is the maximum of all $t_j$'s, and so it divides $p^k$. 
But since $\Gamma$ is a cyclic quotient of $G$ of order $p^k$,
we obtain (iii). 
\end{proof}

\section{Introducing the group of elliptic units}\label{sec_2}
For the rest of the paper, we fix once and for all an embedding $\ov{\Q}\subseteq \C$. In particular, the inclusion 
$K\subseteq \C$ singles out one of the two embeddings of $K$ into $\C$. 
For any subset $\emptyset\ne J\subseteq I$, we let $f_J$ be the least positive integer in $\mathfrak{m}_J$,  
\begin{equation}\label{Definice_wJ}
w_{J}=|\{\zeta\in\mu_K;\,\zeta\equiv1\,(\operatorname{mod}\mathfrak{m}_J)\}|,
\end{equation}
so that $w_J$ divides $w_K:=|\mu_K|$, and we let
\begin{equation}\label{DefiniceEtaF}
\eta_J=
\N_{K_{J}/F_J}(\varphi_{\mathfrak{m}_J})^{w_Kf_I/(w_{J}f_J)}\in\mathcal{O}_{F_J},
\end{equation}
where $\varphi_{\mathfrak{m}_J}$ is defined as in \cite[Definition 2 on page 5]{H}. We would like to point out that the definition 
of $\varphi_{\mathfrak{m}_J}$, as a complex number, uses implicitly the fact that $K$ is included in $\C$.

For a finite abelian extension $M/F$ and a prime ideal $\mk{p}$ of $F$ which is unramified in $M/F$, we use the Artin symbol
$\bigl(\frac{M/F}{\mk{p}}\bigr)\in\Gal(M/F)$ to denote the Frobenius automorphism of $\mk{p}$ in the relative extension $M/F$. 

For any $j\in I$, we let $\lambda_j\in G=\Gal(F_I/K)$ be the unique automorphism such that
$\lambda_j\big|_{F_{I-\{j\}}}=\bigl(\frac{F_{I-\{j\}}/K}{\wp_j}\bigr)$ and $\lambda_j\big|_{F_j}=1$.
The next lemma will be used in the proof of Theorem~\ref{ExistenceAlfaNu} and also in Section~\ref{ZvetseniGrupy}.

\begin{lem}\label{neat3}
For each $j\in I$, choose a prime $\mk{P}_j$ of $L$ above 
$\wp_j$. Then the inertia group of $\mk{P}_j$ is $\langle \sigma^{p^k/\ram_j}\rangle=\Gal(L/L\cap F_{I-\{i\}})$ and the decomposition
group of $\mk{P}_j$ is $\langle\sigma^{n_j}\rangle=\langle\lambda_j\big|_{L},\sigma^{p^k/\ram_j}\rangle$.
\end{lem}
\begin{proof}
The first part was already proved in Corollary~\ref{neat}(i). It thus follows that 
the maximal subextension of $L/K$ which is unramified 
at $\wp_j$ is $L\cap F_{I-\{j\}}/K$. In particular, the Frobenius automorphism of $\wp_j$ 
in $L\cap F_{I-\{j\}}/K$ is equal to 
$\bigl(\frac{L\cap F_{I-\{j\}}/K}{\wp_j}\bigr)=\lambda_j\big|_{L\cap F_{I-\{j\}}}$ and thus
$\langle\lambda_j\big|_{L},\sigma^{p^k/t_j}\rangle$ is equal to the 
decomposition group of $\mk{P}_j$. Moreover, by definition of $n_j$, we have
$[\Gamma:\langle\lambda_j\big|_{L},\sigma^{p^k/t_j}\rangle]=n_j$. Finally, since 
$\langle\sigma^{\dec_j}\rangle$ is the only subgroup of $\Gamma$ of index $\dec_j$, this forces
$\langle\sigma^{\dec_j}\rangle=\langle\lambda_j\big|_{L},\sigma^{p^k/t_j}\rangle$.
\end{proof}

The algebraic numbers $\eta_J$ defined in \eqref{DefiniceEtaF} satisfy the following norm relations which 
can be derived from \cite[Proposition~3 on page 5]{H}: for each $J \subseteq I$ and each $j\in I$ such that $\{j\}\subsetneq J$,
\begin{equation}\label{Relace}
\N_{F_J/F_{J-\{j\}}}(\eta_J)=\eta_{J -\{j\}}^{1-\lambda_j^{-1}},
\end{equation}
and for each $j \in I$,
\begin{equation}\label{JinaRelace}
\N_{F_j/K}(\eta_{\{j\}}) = \N_{H/K}\Bigl(\frac{\Delta(\mathcal{O}_K)}{\Delta(\wp_j)}\Bigr)^{f_I},
\end{equation}
where $\Delta$ is the discriminant Delta function which appears in Section 2.1 of \cite{H}.
It follows from \cite[Proposition 1 on page 3]{H} that $\N_{F_j/K}(\eta_{\{j\}})$ generates
the ideal $\wp_j^{12hf_I}=(\pi_j\mathcal{O}_K)^{12f_I}$, hence
\begin{equation}\label{NormaEtaDole}
\N_{F_j/K}(\eta_{\{j\}})=\xi_j\pi_j^{12f_I},
\end{equation}
for some $\xi_j\in \mu_K$.

The next lemma gives an exact description of the roots of unity in $F_I$. In particular, it will allow us in the sequel to replace $\mu_{F}$ by
$\mu_{K}$ for any subfield $K\subseteq F\subseteq F_I$. 
\begin{lem}\label{OdmocninyZ1}
We have $\mu_{F_I}=\mu_K$.
\end{lem}
\begin{proof}
We do a proof by contradiction. 
Let $\zeta$ be a root of unity in $F_I$ which is not in $K$. In particular, we must have 
$2|[\Q(\zeta):\Q]$ and $[K(\zeta):K]>1$. Using the fact that $p$ is odd, we see that $[K(\zeta):\Q]$ is equal to twice a power of $p$
which implies that $K$ is the only quadratic subfield of $K(\zeta)$. Since $\Q(\zeta)\subseteq K(\zeta)$ and $\Q(\zeta)$ contains
at least one quadratic subfield, we also deduce that
(i) $K$ is the only quadratic subfield of $\Q(\zeta)$ and (ii) $K(\zeta)=\Q(\zeta)$. From (i), it follows that 
there is exactly one prime, say $\ell$, which ramifies in $\Q(\zeta)/\Q$ and that its ramification is total. In particular,
since $K\subseteq\Q(\zeta)\subseteq F_{I}$ and $[\Q(\zeta):K]>1$, the prime $\ell$ must also ramify in $[F_{I}:K]$.
From Corollary \ref{neat} (ii), we know that $F_I/L$ is unramified, and therefore, $\ell$ must also ramify in $L/K$. 
We thus have shown the existence of a rational prime $\ell$ which ramifies in both $K/\Q$ and $L/K$; this contradicts our initial assumptions
on the ramification of the extension $L/\Q$.
\end{proof} 

\begin{defin}\label{tup}
We define the group of elliptic numbers $\GrEllNum_{F_I}$ of $F_I$ to be the $\Z[G]$-submodule of $F_I^\times$ generated by the group of roots of unity 
$\mu_{F_{I}}$ ($=\mu_K$ by Lemma \ref{OdmocninyZ1}) and by $\eta_J$ for all $\emptyset\ne J\subseteq I$. 
The group of elliptic units $\GrEllUn_{F_I}$ of $F_I$ is defined as the intersection $\GrEllUn_{F_I}=\GrEllNum_{F_I}\cap\mathcal{O}_{F_I}^\times$.
The group of elliptic numbers $\GrEllNum_L$ of $L$ is defined as the $\Z[\Gamma]$-submodule of $L^\times$ generated by the group of roots of unity $\mu_{L}$ ($=\mu_{K}$) 
and by $\N_{F_J/F_J\cap L}(\eta_J)$ for all $\emptyset\ne J\subseteq I$. Finally, the group of elliptic units $\GrEllUn_L$ of $L$ is defined
as the intersection $\GrEllUn_L=\GrEllNum_L\cap\mathcal{O}_{L}^\times$.
\end{defin}

Let $M$ be a finite abelian extension of $K$. In \cite[Definition 3 on page~7]{H}, Oukhaba's introduced a group of units in $\ca{O}_M$ 
which we denote by $C_M$. The groups $\ca{C}_{F_I}$ and $\ca{C}_{L}$ which appear in Definition \ref{tup} differ slightly from the
groups $C_{F_I}$ and $C_L$, respectively. Using the key fact that $F_I\cap H=K$ one may check that
$C_{F_I}=\langle \mu_K\cup\{\epsilon^h:\epsilon\in \ca{C}_{F_I}\}\rangle$. Similarly, since  $L\cap H=K$, one may also check that
$C_{L}=\langle \mu_K\cup\{\epsilon^h:\epsilon\in \ca{C}_{L}\}\rangle$.
The two previous equalities will be used in the proof of the following lemma:

\begin{lem}\label{EliptickeJednotky}
(i) The indices of $\GrEllUn_{F_I}$ in $\mathcal{O}_{F_I}^\times$ and $\GrEllUn_L$ in $\mathcal{O}_{L}^\times$ are finite and are 
given explicitly by 
\begin{align*}
[\mathcal{O}_{F_I}^\times:\GrEllUn_{F_I}]&=(12w_Kf_I)^{[F_I:K]-1}\cdot\frac{h_{F_I}}{h},\\
[\mathcal{O}_{L}^\times:\GrEllUn_L]&=(12w_Kf_I)^{[L:K]-1}\cdot\frac{h_{L}}{h\cdot[L:\widetilde L]},
\end{align*}
where $h_{F_I}$, $h_L$, and $h$ are the class numbers of $F_I$, $L$, and $K$, respectively, and $\widetilde L$ is a 
maximal subfield of $L$ containing 
$K$ such that $\widetilde L/K$ is ramified in at most one prime ideal of $K$. Note that such a field $\widetilde{L}$ is unique (and thus well-defined)
since $\Gamma=\Gal(L/K)$ is a cyclic group of a prime power order.

(ii) For any $\beta\in\GrEllNum_{F_I}$ we have $\beta\in\GrEllUn_{F_I}$ if and only if $\N_{F_I/K}(\beta)\in\mu_K$.
\end{lem}
\begin{proof}
It follows from \cite[Theorem 1]{H} that Oukhaba's group of elliptic units is of finite index in the full group of units, 
and so, from the discussion before Lemma~\ref{EliptickeJednotky}, we obtain that $[\ca{C}_{F_I}:C_{F_I}]=h^{[F_I:K]-1}$ and $[\ca{C}_{L}:C_{L}]=h^{[L:K]-1}$.
For a finite abelian extension $F/K$, an index formula for $[\mathcal{O}_{F}^\times:C_{F}]$ is given in \cite[Theorem 1]{H}.
It is formed by the product of four quotients, which we write here, using Oukhaba's notation:
\begin{align}\label{tip}
[\mathcal{O}_{F}^\times:C_{F}]=\frac{(12w_Kf_Ih)^{[F:K]-1}}{w_{F}/w_K}\cdot\frac{h_{F}}{h}
\cdot\frac{\prod_{\mathfrak{p}}[F\cap K_{\mathfrak{p^\infty}}:F\cap H]}{[F:F\cap H]}\cdot\frac{(R_F:U_F)}{d(F)}.
\end{align}
The two formulae in (i) follow from 
Lemma~\ref{OdmocninyZ1} and an explicit
computation of the third and the fourth quotient in \eqref{tip} when $F=F_I$ and $F=L$. 
Let us start by computing the third quotient.
The product is taken over all prime ideals $\mathfrak{p}$ of $K$, and $K_{\mathfrak{p^\infty}}$ means the union of the ray class fields of 
$K$ of modulus $\mathfrak{p}^n$ for all positive integers $n$. 
Since $[F_I:K]$ and $[L:K]$ are powers of $p$ and $p\nmid h$, we have
$F_I\cap H=L\cap H=K$. 
Moreover, by definition of $F_I$, we have $F_I\cap K_{\mathfrak{p^\infty}}=F_{\{j\}}$ if $\mathfrak{p}=\wp_j$ 
and $F_I\cap K_{\mathfrak{p^\infty}}=F_I\cap K_{\mathfrak{p^\infty}}\cap H=K$ if $\mathfrak{p}\notin\{\wp_1,\dots,\wp_s\}$. Combining the previous two observations with
Proposition~\ref{LemmaPrimySoucin}, we obtain that the third quotient is equal to $1$ when $F=F_I$. In the case where
$F=L$, the definition of $\wt{L}$ readily implies that $\prod_{\mathfrak{p}}[L\cap K_{\mathfrak{p^\infty}}:K]=[\widetilde L:K]$,
so that the third quotient is equal to $\frac{1}{[L:\wt{L}]}$. Let us now handle the fourth quotient in \eqref{tip}. It follows 
from \cite[Theorem 5.4]{Si} and Proposition~\ref{LemmaPrimySoucin}
that $(R_F:U_F)=1$ if $F=F_I$. Similarly, it follows from  \cite[Theorem 5.3]{Si} that $(R_F:U_F)=1$
if $F=L$. Finally $d(F_I)=d(L)=1$ by \cite[Remark 2]{H}. 

Let us prove (ii). Let $\beta\in\GrEllNum_{F_I}$.
By \cite[Corollary~2 on page 5]{H} we know that $\eta_J\in\mathcal{O}_{F_J}^\times$ if $|J|>1$ and by \eqref{NormaEtaDole} we know that 
$\eta_{\{j\}}\in\mathcal{O}_{F_j}$ is a generator of a power of the only prime of $F_j$ above $\wp_j$ which ramifies totally in $F_j/K$. 
Hence for any $\tau\in\Gal(F_j/K)$, $\eta_{\{j\}}^{1-\tau}\in\mathcal{O}_{F_j}^\times$. 
Therefore, there is $\gamma\in\GrEllUn_{F_I}$ and $c_1,\dots,c_s\in\Z$ such that $\beta=\gamma\cdot\prod_{j=1}^s\eta_{\{j\}}^{c_j}$. 
Since $\wp_1,\dots,\wp_s$ are different prime ideals, the elliptic numbers $\eta_{\{1\}},\dots,\eta_{\{s\}}$ are multiplicatively independent. 
Hence $\beta\in\GrEllUn_{F_I}$ if and only if $c_1=\dots=c_s=0$. 
Using \eqref{NormaEtaDole} we see that $\N_{F_I/K}(\beta)=\xi\cdot\prod_{j=1}^s\pi_j^{12f_I[F_I:F_j]c_j}$ for some $\xi\in\mu_K$ and 
the lemma follows due to the fact that $\pi_1,\dots,\pi_s$ are multiplicatively independent.
\end{proof}

Recall that $G=\Gal(F_I/K)$.
In \cite{SM}, a $\Z[G]$-module $U$ was introduced which depended solely on the following set of parameters: 
$T_1,\dots,T_v$ and $\lambda_1,\dots,\lambda_v$.
(Warning: Here the module $U$ has a different meaning than in the proof of Lemma~\ref{EliptickeJednotky}, where we used the notations of \cite{H} and \cite{Si}.)
In our situation we put $v=s$ and we set $T_j=\Gal(F_I/F_{I-\{j\}})$ and $\lambda_j\in G$ to be the automorphism defined in the beginning of Section~\ref{sec_2} for each $j=1,\dots,s$. 
For our purpose, it is enough to recall that $U$ was defined explicitly as a certain $\Z[G]$-submodule of $\Q[G]\oplus \Z^s$,
with the following set of $\Z[G]$-generators $U=\langle\rho_J;\,J\subseteq I\rangle_{\Z[G]}$.
Here each $\Z$ summand in $\Q[G]\oplus \Z^s$ is endowed with the trivial $G$-action and 
each element of the standard basis of $\Z^s$ is denoted by $e_j$ (for $j\in I$). Note that by construction $U$
is a finitely generated $\Z$-module with no $\Z$-torsion which implies that $U$ is a free $\Z$-module of finite rank.

The next lemma describes the $\Z[G]$-module structure of $\GrEllNum_{F_I}$ in terms of the $\Z[G]$-module $U$.
For any subset $A\subseteq G$, we let $s(A)=\sum_{a\in A}a\in\Z[G]$.
\begin{lem}\label{Izomorfismus} 
The $\Z[G]$-modules $\GrEllNum_{F_I}/\mu_K$ and $U/(s(G)\Z)$ are isomorphic. More precisely, if we set $\Psi(\eta_J)=\rho_{I-J}$ 
for each $J\subseteq I$, $J\ne\emptyset$, and $\Psi(\mu_K)=0$, then it defines a $\Z[G]$-module homomorphism $\Psi:\GrEllNum_{F_I}\to U$
which satisfies $\ker\Psi=\mu_K$ and $U=\Psi(\GrEllNum_{F_I})\oplus(s(G)\Z)$.
\end{lem}

\begin{proof} It follows from the $\Z[G]$-module presentation of $U$ given in \cite[Corollary~1.6(ii)]{SM}, and the observation that the generator $\rho_I=s(G)$ does not appear 
in the relation \cite[(1.10)]{SM}, that $U=\langle\rho_J;\,J\subsetneq I\rangle_{\Z[G]}\oplus(s(G)\Z)$.
Hence, there exists an embedding of $\Z[G]$-modules $\iota:U/(s(G)\Z)\to U$ such that 
$\im \iota=\langle\rho_J;\,J\subsetneq I\rangle_{\Z[G]}$. In order to define the map $\Psi:\GrEllNum_{F_I}\to U$, 
it is preferable to start by defining its \lq\lq inverse\rq\rq. We define a map
$\Phi:U\rightarrow P_{F_I}$ by setting
$$
\Phi(\rho_J)=\eta_{I-J} \;\;\text{ for each $J\subsetneq I$}\;\;\; 
\mbox{and\; $\Phi(\rho_I)=0$}.
$$
We claim that $\Phi$ is a well-defined $\Z[G]$-module homomorphism whose image together with $\mu_K$ generates $P_{F_I}$. 
Indeed, this follows directly from the $\Z[G]$-module presentation of $U$ given in loc.cit. 
and the norm relation \eqref{Relace}. Since $\Phi(s(G))=0$ and $\langle \Phi(U),\mu_K\rangle=P_{F_I}$, it follows that $\Phi$ induces a surjective $\Z[G]$-module homomorphism
$\wt{\Phi}:U/(s(G)\Z)\rightarrow P_{F_I}/\mu_{K}$. 
Note that $U$ (so a fortiori $U/(s(G)\Z)$ which is embedded in $U$ via $\iota$) and $P_{F_I}/\mu_K$ have no $\Z$-torsion.
Therefore 
in order to show that $\wt{\Phi}$ is a $\Z[G]$-module isomorphism, it is enough to prove that
\begin{align}\label{tij}
\rank_{\Z}(U/(s(G)\Z))=\rank_{\Z}(P_{F_I}/\mu_{K}).
\end{align}
Let us prove \eqref{tij}. Since the prime ideals $\wp_1,\dots,\wp_s$ are distinct, the numbers $\pi_1,\dots,\pi_s$ are multiplicatively independent over $\Z$, 
and Lemma~\ref{EliptickeJednotky} implies that
\begin{align}\label{tij2}
\rank_{\Z}(\GrEllNum_{F_I})=s+\rank_{\Z}(\GrEllUn_{F_I})=s+\rank_{\Z}(\mathcal{O}_{F_I^\times})=s+\tfrac12[F_I:\Q]-1.
\end{align}
Moreover, it follows from \cite[Remark~1.4]{SM} that $\rank_{\Z}(U)=|G|+s$ which, when combined with \eqref{tij2}, proves \eqref{tij}. Finally, we define the map
$\Psi$ as the composition of the three maps 
\begin{align}\label{caf}
P_{F_{I}}\rightarrow P_{F_I}/\mu_K\stackrel{\wt{\Phi}^{-1}}{\rightarrow} U/s(G)\Z\stackrel{\iota}{\rightarrow} U,
\end{align}
where the first map is the natural projection. This proves the existence of $\Psi$ with the desired properties.
\end{proof}

\section{A nontrivial root of an elliptic unit}\label{Odmocnina}

We call the element
\begin{equation}\label{DefiniceEtaL}
\eta=\N_{F_I/L}(\eta_{I})
\end{equation}
the {\it top generator} 
of both the group of elliptic numbers $\GrEllNum_L$ of $L$
and
of the group of elliptic units $\GrEllUn_L$ of $L$.
The aim of this section is to take a nontrivial root \lq\lq $\sqrt[y]{\eta}$ \rq\rq\; of $\eta$ 
(where the root exponent $y$ is a group ring element in $\Z[\Gamma]$) 
such that $\sqrt[y]{\eta}\in L$. We define $\HH=\Gal(F_I/L)\subseteq \Gal(F_I/K)=G$, so that  $\Gamma=\langle\sigma\rangle\cong G/\HH$. 

\begin{lem}\label{Prunik_P_a_L} 
An elliptic number $\beta\in\GrEllNum_{F_I}$ belongs to $L$ if and only if $\Psi(\beta)$ is fixed by $\HH$, i.e.,
$\Psi(\GrEllNum_{F_I})^\HH=\Psi(\GrEllNum_{F_I}\cap L)$,
where $\Psi$ is the $\Z[G]$-module homomorphism introduced in Lemma~\ref{Izomorfismus}.
\end{lem}
\begin{proof}
Let $\beta\in \GrEllNum_{F_I}$. On the one hand, if $\beta\in L$ then $\beta^{\tau-1}=1$ for all $\tau\in \HH$, and so $(\tau-1)\Psi(\beta)=0$, 
which means $\Psi(\beta)\in\Psi(\GrEllNum_{F_I})^\HH$. On the other hand, if $\Psi(\beta)\in\Psi(\GrEllNum_{F_I})^\HH$ then, for any $\tau\in \HH$, 
we have $(\tau-1)\Psi(\beta)=0$ and so $\beta^{\tau-1}=\xi\in\ker(\Psi)=\mu_K$. Note that $\tau^{p^\exponent}=1$ and $\xi^\tau=\xi$. Therefore,
applying $1+\tau+\dots+\tau^{p^\exponent-1}$ to the equality $\beta^{\tau-1}=\xi$ we find that $1=\xi^{p^\exponent}$. Finally, since
$p\nmid|\mu_K|$, we must have $\xi=1$, and therefore $\beta\in L$.
\end{proof}

Recall from Section \ref{bit}, that $\dec_i$ was defined as the index of the decomposition group of the ideal $\mathfrak{P}_i\subseteq L$ 
in $\Gamma$. Without lost of generality, 
we can suppose that
\begin{equation}\label{Serazeni}
\dec_1\le\dec_2\le\ldots\le\dec_s\qquad\text{and we set}\qquad\dec=\dec_s=\max\{\dec_i;\,i\in I\}.
\end{equation}
Since $p|t_s$ we have $n|p^{k-1}$ and it follows from Corollary \ref{neat}(iii) that we can suppose that $t_1=p^k$ and so $\dec_1=1$.
Let $L'$ be the unique subfield of $L$ containing $K$ such that $[L':K]=\dec$. Note that $\langle\sigma^n\rangle=\Gal(L/L')$
and that $\wp_s$ splits completely in $L'/K$. 

We may now state the main result of this section.
\begin{thm}\label{ExistenceAlfaNu}
There is a unique $\alpha\in  L$ such that $\N_{L/L'}(\alpha)=1$ and such that the elliptic unit $\eta$ defined by \eqref{DefiniceEtaL} 
satisfies $\eta=\alpha^{y}$, where $y=\prod_{i=2}^{s-1}(1-\sigma^{\dec_i})$ (if $s=2$ the empty product is taken to mean $1$). 
This $\alpha$ is an elliptic unit of $F_I$, so that $\alpha\in\GrEllUn_{F_I}\cap L$.
Moreover, there is $\gamma\in L^\times$ such that $\alpha=\gamma^{1-\sigma^\dec}$.
and $\N_{L/K}(\bar\nu)\in\langle \mu_K\cup\{\pi_1,\dots,\pi_s\}\rangle$.
\end{thm}

\begin{rem} \label{Vypocet} 
Colloquially we can say that Theorem~\ref{ExistenceAlfaNu} proves the existence of a $y$-th root of the top generator $\eta$ of $\ca{C}_L$ 
which lies in $\GrEllUn_{F_I}\cap L$, where the root exponent $y$ is an element of the group ring $\Z[\Gamma]$. 
In general, even though $y$ is not an integer, it is still possible
to compute $\alpha$ explicitly as a $p$-power root of a specific elliptic unit constructed from the conjugates of $\eta$.
Indeed, for each $j=1,\dots,s$, define the group ring elements 
$$
N_{\dec_j}=\sum_{i=1}^{p^k/\dec_j}\sigma^{i\dec_j}\qquad\text{and}\qquad \Delta_{\dec_j}=\sum_{i=1}^{(p^k/\dec_j)-1}i\sigma^{i\dec_j}.
$$ 
In particular, we have $(1-\sigma^{\dec_j})N_{\dec_j}=0$ and $(1-\sigma^{\dec_j})\Delta_{\dec_j}=N_{\dec_j}-\frac {p^k}{\dec_j}$.

Note also that the relative norm operator $\N_{L/L'}$ corresponds to the group ring element $N_n$. From 
Theorem~\ref{ExistenceAlfaNu}, we know that $\eta=\alpha^{y}$. Moreover, for all $j\in\{1,\ldots,s\}$, we also have  $\alpha^{N_{\dec_j}}=1$ since
$1=\N_{L/L'}(\alpha)=\alpha^{N_n}$. Consequently, we find that
\begin{align}\label{Urcen}
\eta^{\prod_{i=2}^{s-1}\Delta_{\dec_i}}=\alpha^{\prod_{i=2}^{s-1}(N_{\dec_i}-( p^k/{\dec_i}))}=
\alpha^{(-1)^s\prod_{i=2}^{s-1}(p^k/{\dec_i})}=\alpha^{(-1)^sr},
\end{align}
where $r=\prod_{i=2}^{s-1}\frac {p^k}{\dec_i}$ is a power of $p$, and therefore
\begin{equation}\label{UrceniAlfa}
\alpha^r=\eta^{(-1)^s\prod_{i=2}^{s-1}\Delta_{\dec_i}}.
\end{equation}
\end{rem}

To prove Theorem~\ref{ExistenceAlfaNu} we shall use the following proposition:
\begin{prop}\label{NulovostExtDelitelnost}
Let $f$ be a polynomial in $\Z[X]$, $f\notin\{0,\pm1\}$, and let $A=\Z[X]/f\Z[X]$. 
Let $\cal M$ be a finitely generated $A$-module without $\Z$-torsion. Then
\begin{enumerate}[(i)]
\item $\operatorname{Ext}^1_{A}({\cal M},A)=0$.
\item Let $y$ be a nonzerodivisor in $A$, and let $x\in {\cal M}$. Then $x\in y{\cal M}$ if and only 
if for all $\varphi\in\Hom_A(\cal M,A)$ we have $\varphi(x)\in yA$.
\end{enumerate} 
\end{prop}

\begin{proof}
This is \cite[Proposition~6.2]{Eigen}.
\end{proof}

\begin{proof}[Proof of Theorem~\ref{ExistenceAlfaNu}]
If $s=2$ then $y=1$ and therefore the equality $\eta=\alpha^y$ trivially holds true with $\alpha=\eta$.
If $s>2$ we always have that  $y$ is a {\it zerodivisor} in $\Z[\Gamma]$, so that one cannot apply 
directly Proposition~\ref{NulovostExtDelitelnost};  hence we shall work in an appropriate quotient of 
$\Z[\Gamma]$ where the image of $y$ is a nonzerodivisor.
Let $N_\dec=\sum_{i=1}^{p^\exponent/\dec}\sigma^{i\dec}$, so that $N_\dec$  can be understood as the norm operator from $L$ to $L'$.
Let $R=\Z[\Gamma]/N_\dec\Z[\Gamma]$ and let $\gamma:R\to(1-\sigma^\dec)\Z[\Gamma]$ be the isomorphism of $\Z[\Gamma]$-modules given by
the multiplication by $1-\sigma^\dec$, i.e.\ $\gamma (x+N_\dec\Z[\Gamma])=(1-\sigma^\dec)x$. 
Let 
\begin{align}\label{fati}
\mm=\{x\in \Psi(\GrEllNum_{F_I})^\HH;\,N_\dec x=0\},
\end{align}
where $\Psi$ is the map which appears in Lemma \ref{Izomorfismus}. It is an $R$-module and since $\mm\subseteq U$, it has no $\Z$-torsion. 
Using both \eqref{DefiniceEtaL} and the norm relation \eqref{Relace}, we obtain
\begin{equation}\label{PsiEta}
\Psi(\eta)=\Psi(\N_{F_I/L}(\eta_I))=s(\HH)\Psi(\eta_I)=s(\HH)\rho_\emptyset,
\end{equation}
where
$s(\HH)=\sum_{\tau\in\HH}\tau\in\Z[G]$, and
\begin{equation}\label{NormaEta}
\N_{L/L'}(\eta)=\N_{F_I/L'}(\eta_I)=\N_{F_{\{1,\dots,s-1\}}/L'}(\eta_{\{1,\dots,s-1\}})^{1-\lambda_s^{-1}}=1,
\end{equation}
where the last equality follows from the fact that the restriction of $\lambda_s$ to $L'$ is trivial since $\wp_s$ splits completely in $L'/K$. In particular,
it follows from \eqref{PsiEta} and \eqref{NormaEta} that $\Psi(\eta)=s(\HH)\rho_\emptyset\in\mm$.

Note that the natural $\Z[\Gamma]$-module structure on $\mm$ is compatible with its $R$-module structure via the natural projection
map $\Z[\Gamma]\rightarrow R$. In particular,
since $U^\HH=\Psi(\GrEllNum_{F_I})^\HH\oplus(s(G)\Z)$ (from Lemma~\ref{Izomorfismus}),
we may view $\mm$ as a $\Z[\Gamma]$-submodule of $U^{\HH}$. We claim that $U^\HH/\mm$ has no $\Z$-torsion. Indeed,
suppose that $x\in U^\HH$ satisfies $cx\in\mm$ for a positive integer $c$. Then $c(N_n x)=N_n(cx)=0$. Since
$U$ has no $\Z$-torsion, this implies that $N_n x=0$, and hence $x\in\mathcal M$.

To each $R$-linear map $\psi\in\Hm {R}{\mm}$ we may associate
the $\Z[\Gamma]$-linear map $\gamma\circ\psi\in\Hm{\Z[\Gamma]}\mm$. Now we fix such a $\psi$. We aim at proving that
$\psi(s(B)\rho_{\emptyset})\in yR$ (see the relation \eqref{rab2} below).
Note that it makes sense to apply $\psi$ to $s(\HH)\rho_\emptyset$ since it was proved earlier that $s(\HH)\rho_\emptyset\in\mm$.

Now, set $f=X^{p^\exponent}-1$ in Proposition~\ref{NulovostExtDelitelnost} so that $A=\Z[X]/f\Z[X]\cong \Z[\Gamma]$.
Since $U^\HH/\mm$ has no $\Z$-torsion, it follows from Proposition~\ref{NulovostExtDelitelnost} (i) that
$\Ext^1_{\Z[\Gamma]}({U^\HH/\mm},\Z[\Gamma])=0$. In particular, the vanishing of this $\Ext^1$ implies the existence of
$\varphi\in\Hm{\Z[\Gamma]} {U^\HH}$ such that $\varphi\big|_{\mm}=\gamma\circ\psi$. 
For each $x\in U^\HH$, we define $\upsilon(x)=(1-\sigma)\varphi(x)$, so that $\upsilon \in\Hm{\Z[\Gamma]} {U^\HH}$.
We now want to specialize the formula which appears in \cite[Corollary~1.7(ii)]{SM} to the present situation in order to obtain
the non-trivial relation
\begin{equation}\label{PatriDoIdealu}
\upsilon(s(\HH)\rho_\emptyset)\in\prod_{i=1}^s(1-\sigma^{\dec_i})\Z[\Gamma].
\end{equation}
The relation \eqref{PatriDoIdealu} is a direct consequence of the formula in loc.cit.\ and the following two observations:
\begin{enumerate}[(i)]
 \item For all $i\in I$, $\upsilon (\ram_ie_i)=0$, where $t_i=|T_i|$ with $T_i=\Gal(F_I/F_{I-\{i\}})$. (Note that it makes sense to apply the map $v$ to $\ram_ie_i$ since
 $\ram_i e_i\in U^B$.)
 \item It follows from Lemma \ref{neat3} that the element  $1-\lambda_i\big|_L$ lies in the principal ideal $(1-\sigma^{\dec_i})\Z[\Gamma]$. Similarly,
 for each $\tau\in T_i$ we have that $\tau\big|_L\in\langle\sigma^{{p^k}/{t_i}}\rangle$ by Corollary~\ref{neat}(i), and therefore
 $1-\tau\big|_L\in (1-\sigma^{\dec_i})\Z[\Gamma]$.
\end{enumerate}

Since the multiplication by $1-\sigma$ is injective on $(1-\sigma^\dec)\Z[\Gamma]$, it follows from \eqref{PatriDoIdealu} that
\begin{align}\label{rab}
\gamma\circ\psi(s(\HH)\rho_\emptyset)=\varphi(s(\HH)\rho_\emptyset)\in\prod_{i=2}^s(1-\sigma^{\dec_i})\Z[\Gamma].
\end{align}
Furthermore, it follows from \eqref{rab} and the fact that $\gamma$ is an $R$-module isomorphism 
that
\begin{align}\label{rab2}
\psi(s(\HH)\rho_\emptyset)\in \prod_{i=2}^{s-1}(1-\sigma^{\dec_i})R=yR,
\end{align}
where $y=\prod_{i=2}^{s-1}(1-\sigma^{\dec_i})$. We thus have proved that for each $\psi\in\Hm {R}{\mm}$ the relation \eqref{rab2} holds true.

Now set $f=\sum_{i=1}^{p^\exponent/\dec}X^{(i-1)n}$. Since $n|p^{k-1}$ it follows that $f\notin\{0,1,-1\}$; we may thus apply
Proposition~\ref{NulovostExtDelitelnost} with $f$ so that $A=\Z[X]/f\Z[X]\cong R$.
Combining \eqref{rab2} with the observation that $y$ is a nonzerodivisor in $R$ (since the roots of $X^n-1$ are distinct from the roots of $f$), it follows from   
Proposition~\ref{NulovostExtDelitelnost} (ii) that there exists
an element $\delta \in\mm$ such that $y\delta =s(\HH)\rho_\emptyset=\Psi(\eta)$. In particular, since $\delta\in\mm$,
we have $\delta \in\Psi(\GrEllNum_{F_I})^\HH$ and $N_\dec\delta =0$.

By Lemma~\ref{Prunik_P_a_L}, there exists $\alpha'\in \GrEllNum_{F_I}\cap L$ (uniquely defined modulo $\mu_K$) such that $\delta =\Psi(\alpha')$.
We have $\Psi(\N_{L/L'}(\alpha'))=N_\dec\Psi(\alpha')=N_\dec\delta=0$, and so $\N_{L/L'}(\alpha')=\xi\in\mu_K$ by Lemma~\ref{Izomorfismus}. Since $p\nmid|\mu_K|$, 
there is $\xi'\in\mu_K$ such that $\N_{L/L'}(\xi')=\xi^{-1}$. Now if we set $\alpha=\alpha'\xi'\in \GrEllNum_{F_I}\cap L$ we obtain $\N_{L/L'}(\alpha)=1$ 
while still keeping the condition $\delta =\Psi(\alpha)$. Hence $\Psi(\alpha^y)=y\delta =\Psi(\eta)$
and $\xi''=\alpha^{-y}\eta\in\ker(\Psi)=\mu_K$. We claim that $\xi''=1$ so that $\alpha^y=\eta$. Indeed, it follows from \eqref{NormaEta} that 
$1=\N_{L/L'}(\alpha^{-y}\eta)=(\xi'')^{p^\exponent/\dec}$ and consequently $\xi''=1$ (since $p\nmid\mu_K$). Moreover,
since $\N_{L/K}(\alpha)=1$ it follows from Lemma~\ref{EliptickeJednotky}(ii) that $\alpha$ is an elliptic unit of $F_I$. Notice that $\alpha$ is uniquely 
determined by the three conditions (i) $\alpha\in L$, (ii) $\N_{L/L'}(\alpha)=1$ and (iii) $\alpha^{y}=\eta$. Indeed, if there were two such $\alpha$'s, their 
quotient $\beta\in L$ would satisfy $\beta^{y}=1$. Similarly to what we did in \eqref{Urcen}, we may apply the 
group ring element $\prod_{j=2}^{s-1}\Delta_{n_j}$ to the equality $\beta^{y}=1$ to 
find that $1=\beta^{r}$ (this uses (ii)) where $r$ is some power of $p$. Since $p\nmid|\mu_L|$ this implies that $\beta=1$.

Finally, applying Hilbert's Theorem 90 to the cyclic extension $L/L'$ implies that there exists a $\gamma\in L^\times$, well-defined up to a 
multiplication by numbers in $(L')^\times$, such that $\alpha=\gamma^{1-\sigma^\dec}$. This concludes the proof. 
\end{proof}

\section{Enlarging the group $\GrEllUn_{L}$ of elliptic units of $L$} \label{ZvetseniGrupy}

We keep the same notation as in the previous sections and we introduce some new one. 
Let us label each subfield of $L$ containing $K$ as follows:
\begin{equation}\label{Mezitelesa}
K=L_0\subsetneq L_1\subsetneq L_2\subsetneq \dots\subsetneq L_k=L.
\end{equation}
In particular, we must have $[L_i:K]=p^i$. 
For each $i=1,\dots,k$ we define
\begin{equation}\label{DefiniceM}
M_i=\bigl\{j\in\{1,\dots,s\};\;\ram_j>p^{k-i}\bigr\}.
\end{equation}
It follows from the definition of $M_i$ that $M_1\subseteq M_2\subseteq\dots\subseteq M_k=\{1,\dots,s\}$, and from the discussion below
\eqref{Serazeni} that $1\in M_1$. 
One may also check using Corollary~\ref{neat}(i) that $j\in M_i$ if and only if $\wp_j$ ramifies in $L_i/K$; in particular,
the conductor of $L_i/K$ is equal to $\mathfrak m_{M_i}$
and so $L_i\subseteq F_{M_i}$ by Proposition~\ref{LemmaPrimySoucin} applied to $L_i/K$ instead of $L/K$.
We define
\begin{equation}
\label{DefiniceEtaLi}
\eta_i=\N_{F_{M_i}/L_i}(\eta_{M_i})\;\;\; \mbox{for $i\in\{1,\ldots,k\}$},
\end{equation}
so that, for example, $\eta_k=\eta\in L=L_k$ is the top generator of  $\GrEllUn_L$,
the group of elliptic units of $L$.
Using the norm relation \eqref{Relace} one may check that $\GrEllUn_L$ is the $\Z[\Gamma]$-module 
generated by $\mu_K$ and by $\eta_1,\dots,\eta_k$.

Before defining the extended group of elliptic units (see Definition \ref{DefiniceAlfa_i} below), we need to fix some more notation.
We fix an index $j\in\{1,\ldots, s\}$ and we let $L_i$ be the largest subfield of $L$ which appears 
in the tower \eqref{Mezitelesa} where $\wp_j$ is unramified; 
the index $i$ is determined by the condition $t_j=p^{k-i}$. Using Lemma~\ref{neat3}, it makes sense to define  $c_j$ as the smallest positive integer such that 
$\sigma^{-c_j\dec_j}\big|_{L_i}=\lambda_j\big|_{L_i}$. Indeed, it follows from the group equality 
$\langle\sigma^{\dec_j}\rangle=\langle\lambda_j\big|_{L},\sigma^{p^k/t_j}\rangle$ in Lemma~\ref{neat3} that
\begin{align}\label{gr_eq}
\langle\sigma^{\dec_j}\rangle/\langle \sigma^{p^k/t_j}\rangle=\langle\lambda_j,{\sigma}^{p^k/t_j}\rangle/\langle\sigma^{p^k/t_j}\rangle.
\end{align}
Note that the quotient group in \eqref{gr_eq} can also be interpreted as the restriction of $\langle\sigma^{\dec_j}\rangle$ to $L_i$.
It follows from \eqref{gr_eq} that $\wp_j$ splits completely in $L_i/K$ if and only if $\frac{p^k}{t_j}=n_j$; in particular, if $\wp_j$ splits completely 
in $L_i/K$ then $c_j=1$ since $\sigma^{n_j}$ lies already in the inertia group of $\mk{P}_j$. If $\wp_j$ does not split completely in $L_i/K$, then 
it follows again from \eqref{gr_eq} that $n_j<\frac{p^k}{t_j}$ and thus $\langle (\sigma|_{L_i})^{n_j}\rangle=
\langle (\lambda_j|_{L_i})\rangle$. In particular, independently of the splitting behavior of $\wp_j$ in $L_i$, we always have that $p\nmid c_j$ and hence $1-\sigma^{c_j\dec_j}$ and 
$1-\sigma^{\dec_j}$ are associated in $\Z[\Gamma]$, i.e.\ each of them divides the other.

Recall that we had chosen an ordering of the ramified primes $\wp_1,\dots,\wp_s$ in the relative extension $L/K$ in such a way that $1=\dec_1\le\dec_2\le\dots\le\dec_s$,
and that this ordering was implicitly assumed in the statement of Theorem~\ref{ExistenceAlfaNu}.
For each index $i\in\{1,\dots,k\}$ such that $|M_i|>1$, Theorem~\ref{ExistenceAlfaNu}, when applied to the extension $L_i/K$, implies
the existence of an elliptic unit $\alpha_i\in\GrEllUn_{F_I}\cap L_i$ and of a number $\gamma_i\in L_i^\times$ such that: 
\begin{enumerate}[(i)]
\item the elliptic unit $\eta_i$ defined in \eqref{DefiniceEtaLi} satisfies 
$\eta_i=\alpha_i^{y_i}$,
\item $\alpha_i=\gamma_i^{z_i}$,
\end{enumerate}
where $z_i=1-\sigma^{c_{\max M_i}\dec_{\max M_i}}$ and $y_i=\prod_{j\in M_i,\ 1<j<\max M_i}(1-\sigma^{c_j\dec_j})$.
In particular, if $|M_i|=2$, we find that $y_i=1$ and $\alpha_i=\eta_i$, since the product is empty. If $i\in\{1,\dots,k\}$ is such that 
$|M_i|=1$ then we set $\gamma_i=\eta_i$ and $\alpha_i=\eta_i^{1-\sigma}$.
\begin{defin}\label{DefiniceAlfa_i}
We define the {\it extended group of elliptic units} $\overline\GrEllUn_L$ to be the
$\Z[\Gamma]$-submodule of ${\mathcal O}^\times_{L}$ generated by $\mu_K$ and by the units $\alpha_1,\dots,\alpha_k$.
\end{defin}

Repeating the arguments of \cite{AnnIII} we can show the following:
\begin{thm}\label{Indexy}
The group of elliptic units $\GrEllUn_L$ of $L$ is a subgroup of 
$\overline\GrEllUn_L$ of index $\bigl[\overline\GrEllUn_L:\GrEllUn_L\bigr]=p^\nu$, where
\begin{equation}\label{DefiniceNu}
\nu=\sum_{j=1}^k\sum_{\mycom{\sst i\in M_j}{\sst 1<i<\max M_j}}\dec_i.
\end{equation}
Moreover, if we let $\varphi_L=\Bigl(\prod_{i=1}^s \ram_i^{\dec_i}\Bigr) \cdot\prod_{j=1}^k p^{-\dec_{\max M_j}}$, which is a power of $p$, then
\begin{equation}\label{PopisNu}
p^\nu=\varphi_L\cdot[L:\widetilde L]^{-1},
\end{equation}
where $\widetilde L$ has the same meaning as in Lemma~\ref{EliptickeJednotky} and
\begin{equation}\label{IndexRozsireneGrupy}
\bigl[{\mathcal O}^\times_L:\overline\GrEllUn_L\bigr]=(12w_Kf_I)^{p^k-1}\cdot\frac{h_L}{h} \cdot\varphi_L^{-1},
\end{equation}
where $h_L$ is the class number of $L$. In particular, if $p>3$ then $p\nmid12w_Kf_I$, and thus it follows  
from \eqref{IndexRozsireneGrupy} that $\varphi_L\mid h_L$.
\end{thm}

\begin{proof}
The proof goes along the same lines as in \cite[Theorem 3.1]{AnnIII}.
The reason why the same algebraic manipulations are possible here (for elliptic units) and in \cite{AnnIII} (for circular units) is given by the fact that 
in both cases 
we work with a module isomorphic to $U/(s(G)\Z)$ (compare Lemma~\ref{Izomorfismus} with \cite[Lemma 1.1]{AnnIII}).
\end{proof}

\begin{rem}\label{key_rem}
The divisibility statement $\varphi_L\mid h_L$ is stronger than what one can get from the mere fact that $F_I/L$ is an unramified 
abelian extension, see Corollary~\ref{neat}(ii). Indeed, \cite[Proposition 3.4]{AnnIII} states that 
we always have $[F_I:L]\mid\varphi_L$ and that $\varphi_L=[F_I:L]$ 
if and only if $\dec_1=\dots=\dec_{s-1}=1$.
\end{rem}

\section{Semispecial numbers}\label{Semi}

We keep the same notation as in the previous sections. In particular,
$\Gamma=\Gal(L/K)\cong\Z/p^k\Z$ and $s$ is the exact number of prime ideals of $K$ which ramify in $L$.
For the rest of the paper, we fix $\M$, a power of $p$, such that $p^{ks}\mid \M$.
For a prime ideal $\fq$ of $K$, recall that $K(\fq)$ denotes the ray class field of $K$ of modulus $\fq$. 
From Artin's Reciprocity Theorem we know that
\begin{equation}\label{ART}
\Gal(K(\fq)/H)\cong(\mathcal{O}_K/\fq)^{\times}/\im(\mu_K),
\end{equation}
where $H$ is the Hilbert class field of $K$. In particular, $\Gal(K(\fq)/H)$
is a cyclic group. We are now ready to define a family of distinguished abelian extensions over $K$ which have a cyclic Galois group of order $m$.
\begin{defin}\label{neat4}
To each prime ideal $\fq$ of $K$ such that $|\mathcal{O}_K/\fq|\equiv1\pmod \M$ we define 
the field $K[\fq]$ to be the (unique) subfield of $K(\fq)$ containing $K$ such that $[K[\fq]:K]=\M$. Moreover, given a finite field
extension $M/K$ we also define $M[\fq]$ to be the compositum of $M$ with $K[\fq]$.
\end{defin}

Note that since $|\mathcal{O}_K/\fq|\equiv1\pmod \M$ and $p\nmid|\mu_K|$, the group $\Gal(K(\fq)/H)$ is cyclic of order divisible
by $\M$. Therefore, since $p\nmid h$, the existence and the uniqueness of the field $K[\fq]$ follows  directly from Lemma~\ref{cyc} applied to 
the triple $(\Gal(K(\fq)/K),\Gal(K(\fq)/H),\M)$.
It is clear that $\Gal(K[\fq]/K)\cong\Z/\M\Z$ 
and one may also check that $K[\fq]/K$ is ramified only at $\fq$ and 
that this ramification is total and tame.

\begin{defin}\label{DefiniceQm}
Let ${\mathcal Q}_\M$ be the set of all prime ideals $\fq$ of $K$ such that
\begin{enumerate}[(i)]
\item $\fq$ is of absolute degree 1, so that $q=|\mathcal{O}_K/\fq|$ is a prime number;
\item $q\equiv1+\M\pmod{\M^2}$;
\item $\fq$ splits completely in $L$;
\item for each $j=1,\dots,s$, the class of $\pi_j$ is an $\M$-th power in $(\mathcal{O}_K/\fq)^\times$.
\end{enumerate}
\end{defin}

Let us make a few observations about the field $K[\mk{q}]$ and also about the fourth condition of Definition \ref{DefiniceQm}.
Note that Artin's Reciprocity Theorem gives slightly more information concerning the isomorphism \eqref{ART}: the class of $\alpha\in\mathcal{O}_K-\fq$ is mapped 
to the automorphism given by the Artin symbol $\bigl(\tfrac{K(\fq)/K}{\alpha\mathcal{O}_K}\bigr)$. 
Since $H\cap K[\fq]=K$, we have $\Gal(H[\fq]/H)\cong\Gal(K[\fq]/K)$ where the isomorphism is given by restriction, and so 
factoring out the $\M$-th powers in \eqref{ART}, we get the following sequence of isomorphisms:
\begin{equation}\label{Artin}
(\mathcal{O}_K/\fq)^\times/\M\stackrel{\cong}{\longrightarrow}
\Gal(H[\fq]/H)\stackrel{\cong}{\longrightarrow}
\Gal(K[\fq]/K),
\end{equation}
where the first map takes the class of $\alpha\in\mathcal{O}_K-\fq$ to
$\bigl(\tfrac{H[\fq]/K}{\alpha\mathcal{O}_K}\bigr)$, and the second map takes $\bigl(\tfrac{H[\fq]/K}{\alpha\mathcal{O}_K}\bigr)$ to its restriction 
$\bigl(\tfrac{K[\fq]/K}{\alpha\mathcal{O}_K}\bigr)$.
Hence combining the observations that $\pi_j\mathcal{O}_K=\wp_j^h$, $p\nmid h$, with the sequence of isomorphisms appearing in \eqref{Artin}, 
we see that the fourth condition (iv) is equivalent to the statement that
\begin{equation}\label{Splitting}
\bigl(\tfrac{K[\fq]/K}{\wp_j}\bigr)=1\qquad\text{for each $j=1,\dots,s$.}
\end{equation}

\begin{defin}\label{semispe}
A number 
$\varepsilon\in L^\times$ is called {\it $\M$-semispecial\/} 
if for all but finitely many $\fq\in{\mathcal Q}_\M$, there exists a unit $\varepsilon_\fq\in{\mathcal O}_{L[\fq]}^\times$ satisfying
\begin{enumerate}[(i)]\itemsep2pt \parskip0pt \parsep0pt
\item $\N_{L[\fq]/L}(\varepsilon_\fq)=1$;
\item If ${\tilde\fq}$ is the product of all primes of $L[\fq]$ above $\fq$, then $\varepsilon$ and $\varepsilon_\fq$ 
have the same image in $({\mathcal O}_{L[\fq]}/ {\tilde\fq})^\times/(\M/p^{k(s-1)})$.
\end{enumerate}
\end{defin}

\medskip

Let us make a few basic observations about the field $L[\fq]$ which appears in Definition \ref{semispe}. 
For each $\mk{q}\in\ca{Q}_m$, we have that $\Gal(K[\fq]/K)\cong \Z/m\Z$, that $\fq$ is totally ramified in $K[\fq]/K$ and that it splits completely in $L/K$.
In particular, we must have that $L[\fq]/L$ is totally ramified at each prime above $\fq$ and that $L\cap K[\fq]=K$. Since $L$ and $K[\fq]$ are linearly disjoint
over $K$, it follows that the two restriction
maps $\Gal(L[\fq]/L)\rightarrow \Gal(K[\fq]/K)$ and $\Gal(L[\fq]/K[\fq])\rightarrow \Gal(L/K)$ are isomorphisms. 
\begin{thm}\label{Semisp}
The elliptic unit $\alpha\in\GrEllUn_{F_I}\cap L$ described in Theorem~\ref{ExistenceAlfaNu} is $\M$-semispecial.
\end{thm}
\begin{proof}
Recall that the elliptic unit $\alpha\in \GrEllUn_{F_I}\cap L$ was obtained in Theorem~\ref{ExistenceAlfaNu} as a $y$-th root of the 
top generator $\eta$ of $\GrEllUn_L$. In order to show that $\alpha$ is $\M$-semispecial, we need to show that 
for almost all primes $\fq\in\mathcal{Q}_\M$, there exists a unit $\varepsilon_{\mk{q}}\in\mathcal{O}_{L[\fq]}^\times$ 
which satisfies the conditions (i) and (ii) of Definition~\ref{semispe} for $\varepsilon=\alpha$.
In order to show that such an $\varepsilon_{\mk{q}}$ exists, we use an approach similar to the one used  
in the proof of Theorem~\ref{ExistenceAlfaNu}. But this time, the role played by $\eta$ in Theorem~\ref{ExistenceAlfaNu}
will be played by $\hat{\eta}=\N_{F_I[\fq]/L[\fq]}(\tilde{\eta}_{I'})$ where $\tilde{\eta}_{I'}$ (to be defined below) is the top generator of $\GrEllNum_{F_I[\fq]}$.

For the rest of the proof we fix a prime $\fq\in\mathcal{Q}_\M$ unramified in $K/\Q$, which does not divide $q_1\cdots q_s$. 
To simplify the notation, we let $\wp_{s+1}=\fq$, $F_{s+1}=K[\fq]$, and $I'=\{1,\dots,s+1\}$.
Again, for any subset $J\subseteq I'$ with $J\neq\emptyset$, we set $F_J=\prod_{j\in J}F_j$,
$\mathfrak{m}_J=\prod_{j\in J}\wp_j$ (the conductor of $F_J$), and 
\begin{align}\label{ti_uni}
\tilde\eta_J=
\N_{K(\mathfrak{m}_J)/F_J}(\varphi_{\mathfrak{m}_J})^{w_Kf_{I'}/(w_Jf_J)},
\end{align}
where $f_J$ and $w_J$ are defined as in \eqref{Definice_wJ} and $\varphi_{\mathfrak{m}_J}$ is defined as in \cite[Definition~2 on page 5]{H}.
If $J\subseteq I$,
this definition does not change the previous meaning of $F_J$ while $\tilde\eta_J=\eta_J^{q}$, where $q=|\mathcal{O}_K/\fq|=\frac{f_{I'}}{f_I}$. 
It follows also from the definitions that $F_I[\fq]=F_{I'}$ and $\mathfrak{m}_{I'}=\fq \mathfrak{m}_I$. 
By the same reasoning as in Lemma~\ref{OdmocninyZ1} we find that $\mu_{F_I[\fq]}=\mu_K$.

Let $G_\fq=\Gal({F_I[\fq]}/K)$ and let $P_{F_I[\fq]}$ be the group of elliptic numbers of $F_I[\fq]$, i.e.\ $P_{F_I[\fq]}$ is the $\Z[G_\fq]$-module 
generated in $F_I[\fq]^\times$ by $\mu_K$ and by $\tilde\eta_J$ for all $J\subseteq I'$, $J\ne\emptyset$. 
Let $U_\fq\subseteq\Q[G_\fq]\oplus\Z^{s+1}$ be the $\Z[G_\fq]$-module defined in \cite{SM} with the following
parameters: $v=s+1$, for each $j\in\{1,\dots,v\}$, $T_j=\Gal(F_{I'}/F_{I'-\{j\}})$ (the inertia group of $\wp_j$ in $G_\fq$), and $\lambda_j\in G_\fq$
is such that the restrictions $\lambda_j\big|_{F_j}=1$ and $\lambda_j\big|_{F_{I'-\{j\}}}=\bigl(\frac{F_{I'-\{j\}}/K}{\wp_j}\bigr)$.

Now, in order to simplify the notation, we choose to make some natural identifications between certain objects: \lq\lq the old ones\rq\rq\, which have 
already appeared in the proof of Theorem \ref{ExistenceAlfaNu} and \lq\lq the new ones\rq\rq\, which appear in the present proof. Consider
the sequence
\begin{equation}\label{identi}
 \Gal({F_I[\fq]}/K[\fq])\subseteq G_\fq\rightarrow G=\Gal(F_I/K),
\end{equation}
where the arrow is given by the restriction map. We decide to identify $\Gal({F_I[\fq]}/K[\fq])$ with $G$ via the above diagram. In particular, the new 
groups $T_i$ defined in the paragraph just above, for $i\ne s+1$, are identified to the old ones, and if we set 
$\HH=\Gal({F_I[\fq]}/L[\fq])$ it is also identified with the old $B$. The assumption that $\fq\in{\mathcal Q}_\M$ also implies that 
the new elements $\lambda_i$, for $i\in I$, are identified to the old ones (by \eqref{Splitting}) and that $\lambda_{s+1}\in \HH$ (since
$\fq$ split completely in $L$). However, 
the $\Z[G]$-generators of $U\subseteq \Z[G]\oplus\Z^s$ cannot be identified, in any meaningful way, to a subset of the $\Z[G_{\mk{q}}]$-generators
of $U_{\mk{q}}\subseteq\Z[G_{\mk{q}}]\oplus\Z^{s+1}$; so we need to distinguish between these two sets of generators.
Recall, in the notation of \cite{SM}, that $U=\langle\rho_J$; $J\subseteq I\rangle_{\Z[G]}$, that the standard basis of $\Z^s$ is denoted by $e_1,\dots,e_s$,
and that $\pi:\Q[G]\oplus\Z^s\to\Q[G]$ is the projection onto 
the first summand. We set $U'=\pi(U)$, so that $U'$ is generated by $\rho'_J=\pi(\rho_J)$. 
We choose to denote the $\Z[G_\fq]$-generators 
of $U_\fq$ by $\tilde\rho_J$, so that $U_\fq=\langle\tilde\rho_J;\,J\subseteq I'\rangle_{\Z[G_\fq]}$, and the standard basis of 
$\Z^{s+1}$ by $\tilde e_1,\dots,\tilde e_{s+1}$. The next lemma gives precise relationships between the modules $U$, $U'$ and $U_\fq$;
for its proof see \cite[Lemma 2.1]{AnnIII}).

\begin{lem}\label{PorovnaniModulu}
Recall that $G$ is viewed as a subgroup of $G_{\fq}$ via \eqref{identi}. There are injective $\Z[G]$-homomorphisms $\chi:U\to U_\fq$ and $\chi':U'\to U_\fq$ defined by
$$
\chi(\rho_J)=\tilde\rho_{J\cup\{s+1\}}\qquad\text{and}\qquad\chi'(\rho'_J)=\tilde\rho_{J},
$$
for each $J\subseteq I$. Moreover, $U_\fq\cong U\oplus\Z\oplus(U')^{\M-1}$ as $\Z[G]$-modules .
\end{lem}

We can apply Lemma~\ref{Izomorfismus} to our present situation which gives us a homomorphism 
$\Psi_\fq:P_{F_I[\fq]}\to U_\fq$ of $\Z[G_\fq]$-modules defined by $\Psi_\fq(\eta_J)=\tilde\rho_{I'-J}$ for each 
$J\subseteq I'$, $J\ne\emptyset$, and $\Psi_\fq(\mu_K)=0$; where $\ker\Psi_\fq=\mu_K$ and $U_\fq=\Psi_\fq(P_{F_I[\fq]})\oplus(s(G_{\mk{q}})\Z)$. 
Let us define 
\begin{align}\label{tupe}
\hat\eta=\N_{F_I[\fq]/L[\fq]}(\tilde\eta_{I'}).
\end{align}
Then we have
\begin{align}\label{beet}
\Psi_\fq(\hat\eta)=s(\HH)\Psi_\fq(\tilde{\eta}_{I'})=s(\HH)\tilde\rho_\emptyset,
\end{align}
and $\Psi_\fq(P_{F_I[\fq]}\cap L[\mk{q}])=\Psi_\fq(P_{F_I[\fq]})^\HH$, where the last equality can be proved along the same lines as Lemma~\ref{Prunik_P_a_L}.
As in \eqref{Serazeni}, we let $\dec=\max\{\dec_i;\,i\in I\}$,   
and as in the proof of Theorem~\ref{ExistenceAlfaNu} we also let 
$N_\dec=\sum_{i=1}^{p^\exponent/\dec}\sigma^{i\dec}$, and
$R=\Z[\Gamma]/N_\dec\Z[\Gamma]$, where now
$\Gamma=\Gal(L[\fq]/K[\fq])=\langle\sigma\rangle$ (here the new $\sigma$ restricts to the old one). 
We also let $\gamma :R\to(1-\sigma^\dec)\Z[\Gamma]$ be the isomorphism of $\Z[\Gamma]$-modules induced by the multiplication by $1-\sigma^n$. 
Note that the group ring element $N_\dec\in \Z[\Gamma]$ corresponds to the norm operator of $L[\fq]/L'[\fq]$, 
where $L'$ is the field defined just before Theorem~\ref{ExistenceAlfaNu}.

One may also check that the set
$$
\mm_\fq=\{x\in \Psi_\fq(P_{F_I[\fq]})^\HH;\,N_\dec x=0\}
$$
is again an $R$-module (so also a $\Z[\Gamma]$-module) without $\Z$-torsion such that $U_\fq^\HH/\mm_\fq$ has no $\Z$-torsion. In particular,
we may apply Proposition~\ref{NulovostExtDelitelnost} with the polynomial $f=X^{p^k}-1$ to deduce that $\operatorname{Ext}^1_{\Z[\Gamma]}({U_\fq^\HH/\mm_\fq},\Z[\Gamma])=0$.
We also have the equalities
\begin{equation}\label{JednaNorma}
\hat\eta^{N_n}=\N_{L[\fq]/L'[\fq]}(\hat\eta)=\N_{F_I[\fq]/L'[\fq]}(\tilde\eta_{I'})=1,
\end{equation}
where $\hat\eta$ is defined in \eqref{tupe} and $\tilde\eta_{I'}$ in \eqref{ti_uni}.
Indeed, the first equality follows from the definition of $N_n$ and the second one follows from \eqref{tupe}.
For the third equality, note that since $\wp_s$ splits completely in $L'$ (by definition of $L'$)
and also in $K[\fq]$ (by \eqref{Splitting}), then it must also split completely in $L'[\fq]$, and
therefore, from the norm relation \eqref{Relace}, the third equality follows. Combining \eqref{JednaNorma} with \eqref{beet}
we obtain 
\begin{equation}\label{craft}
 s(\HH)\tilde\rho_\emptyset\in\mm_\fq.
\end{equation}

To each $R$-linear functional $\psi\in\Hm {R}{\mm_\fq}$, we may associate the map $\gamma\circ\psi$ which can be viewed naturally as an element
of $\Hm {\Z[\Gamma]}{\mm_\fq}$. Hence, because of the vanishing of the $\Ext^1$, for any given $\psi\in\Hm {R}{\mm_\fq}$, there exists a $\varphi\in\Hm{\Z[\Gamma]} {U_\fq^\HH}$ 
such that $\varphi\big|_{\mm_\fq}=\gamma\circ\psi$. 

The restriction of the projection $\pi:\Q[G]\oplus\Z^s\to\Q[G]$ to $U$ gives a surjective map $\pi\big|_U:U\to U'$, which can be composed with 
the map $\chi'$ of Lemma~\ref{PorovnaniModulu}, to give rise to the $\Z[G]$-linear map 
$\chi'\circ\pi\big|_U:U\to U_\fq$. Restricting further the previous map to $U^B$, we obtain the two maps 
$\chi'\circ\pi\big|_{U^\HH}\in\Hom_{\Z[\Gamma]}(U^\HH,U_\fq^\HH)$ and
$\varphi\circ\chi'\circ\pi\big|_{U^\HH}\in\Hm{\Z[\Gamma]} {U^\HH}$.

We have the following relation:
\begin{align}\label{boot}
\varphi(s(\HH)\tilde\rho_\emptyset)=\varphi\circ\chi'\circ\pi(s(\HH)\rho_\emptyset)\in\prod_{i=1}^s(1-\sigma^{\dec_i})\Z[\Gamma]=(1-\sigma)y(1-\sigma^n)\Z[\Gamma],
\end{align}
where $y=\prod_{i=2}^{s-1}(1-\sigma^{\dec_i})$ is defined as in the statement of Theorem~\ref{ExistenceAlfaNu}. 
Indeed, the first equality follows from the facts that $\chi'\circ\pi(\rho_{\emptyset})=\tilde{\rho}_{\emptyset}$
and that $\chi'\circ\pi$ is $\Z[G]$-linear.
The membership relation follows from \cite[Corollary~1.7(ii)]{SM} and the observation that $\pi(t_je_j)=0$ for all $j\in J$ in the same way as \eqref{PatriDoIdealu}.

It follows from \eqref{craft} that the evaluation $\psi(s(\HH)\tilde\rho_\emptyset)$ makes sense for any $\psi\in\Hm {R}{\mm_\fq}$; and
it follows from \eqref{boot}
and the injectivity of $\gamma$ that
$$
\psi(s(\HH)\tilde\rho_\emptyset)\in(1-\sigma)yR.
$$
Since $\psi$ was arbitrary, Proposition~\ref{NulovostExtDelitelnost} implies that there exists $\delta \in\mm_\fq$ such that 
\begin{align}\label{grou}
(1-\sigma)y\cdot\delta =s(\HH)\tilde\rho_\emptyset=\Psi_\fq(\hat\eta).
\end{align} 
Since $\delta\in\mm_\fq$, there exists a $\beta'\in P_{F_I[\fq]}\cap L[\fq]$ such that $\delta=\Psi_\fq(\beta')$ and $\Psi_\fq(\N_{L[\fq]/L'[\fq]}(\beta'))=0$. 
In particular, we have that $\xi=\N_{L[\fq]/L'[\fq]}(\beta')\in\ker(\Psi_{\fq})=\mu_K$. 
Since  $\N_{L[\fq]/L'[\fq]}(\xi)=\xi^{p^k/\dec}$ and $p\nmid|\mu_K|$, there is $\xi'\in\mu_K$ such that $\N_{L[\fq]/L'[\fq]}(\xi')=\xi^{-1}$. 
We set $\beta=\beta'\xi'$, so that $\beta$ satisfies the norm relation $\N_{L[\fq]/
L'[\fq]}(\beta)=1$ while still keeping the equality $\delta=\Psi_\fq(\beta)$. Since $\Psi_\fq(\beta^{(1-\sigma)y})=(1-\sigma)y\delta=\Psi_\fq(\hat\eta)$,
it follows that $\xi''=\beta^{-(1-\sigma)y}\hat\eta\in\ker(\Psi_{\fq})=\mu_K$. We claim that $\xi''=1$. Indeed,  
from \eqref{JednaNorma} we have $1=\N_{L[\fq]/L'[\fq]}(\xi'')=(\xi'')^{p^k/\dec}$, 
and therefore $\xi''=1$. We thus have constructed an elliptic number $\beta\in P_{F_I[\fq]}\cap L[\fq]$ which satisfies the equality $\beta^{(1-\sigma)y}=\hat\eta$.

Now, we would like to show that the elliptic number $\beta$ constructed in the above paragraph
is a unit which satisfies the additional condition $\N_{L[\fq]/L}(\beta)=1$. 
By a similar computation as the one done in Remark~\ref{Vypocet}, we find that
\begin{equation}\label{UrceniBeta}
\beta^{r(1-\sigma)}=\hat\eta^{(-1)^s\prod_{i=2}^{s-1}\Delta_{\dec_i}}\;\;\;\;\;\;\;\mbox{where}\;\;\;\;\;\; r=\prod_{i=2}^{s-1}\frac {p^k}{\dec_i}.
\end{equation}
In particular, applying $\Delta_1$ on each side of the first equality in \eqref{UrceniBeta} and using the norm relation 
$\N_{L[\fq]/L'[\fq]}(\beta)=\beta^{N_n}=1$, we find that 
\begin{align}\label{box}
\beta^{rp^k}=\hat\eta^{(-1)^{s+1}\prod_{i=1}^{s-1}\Delta_{\dec_i}}.
\end{align}
We have 
\begin{align}\label{box2}
 \N_{L[\fq]/L}(\hat\eta)=\N_{F_I[\fq]/L}(\tilde\eta_{I'})=
\N_{F_I/L}(\tilde\eta_{I})^{1-\lambda_{s+1}^{-1}}=1,
\end{align}
where the first equality follows from the definitions of $\hat\eta$ and $\tilde\eta_{I'}$, the second equality from the
norm relations \eqref{Relace}, and the last equality from the fact that $\fq$ splits completely in $L/K$.
Combining \eqref{box} and \eqref{box2}, with the fact that $p\nmid|\mu_K|$, we deduce that $\N_{L[\fq]/L}(\beta)=1$.
From the previous equality we get that $\N_{F_I[\fq]/K}(\beta)=1$, 
and therefore, applying Lemma~\ref{EliptickeJednotky}(ii) we deduce that $\beta$ is a unit.

In order to finish the proof that $\alpha$ is $m$-semispecial, we need to construct a unit
$\varepsilon_\fq\in L[\fq]$ which satisfies the conditions (i) and (ii) of Definition \ref{semispe} for $\varepsilon=\alpha$.
We set $\varepsilon_\fq=\beta^{1-\sigma}$. So far, from what has been proved on $\beta$, we know that
$\varepsilon_{\fq}$ is a unit which satisfies the norm relation (i). By means of 
the next proposition (see Proposition \ref{ChybejiciKongruence} below) we shall prove that $\varepsilon_{\mk{q}}$ and $\alpha$ also satisfy the 
congruence relation (ii). 

Let us recall some of the notation that was  fixed at the beginning of Section \ref{Semi}. The integer
$m$ is a fixed power of $p$, such that $p^{ks}|m$, $\fq$ is a prime ideal of $K$ which lies in the special set $\ca{Q}_m$.
In particular, it follows from the definition of $\ca{Q}_m$ that $\fq$ splits completely in $L/K$, and that the extension $L[\mk{q}]/L$ is cyclic of degree $m$,
and that it is totally ramified at each prime above $\fq$. 

\begin{prop}\label{ChybejiciKongruence}
Let $\mk{q}\in\ca{Q}_m$ be the prime that was fixed during the course of the proof of Theorem \ref{Semisp}, and
let ${\tilde\fq}$ denote the product of all the primes of $L[\fq]$ above $\fq$. Then, there exists a rational prime $\ell\equiv1\pmod\M$ 
such that the following congruence holds true:
\begin{equation}\label{Kongruence}
\hat\eta^{\ell(1-\sigma)}\equiv
(\eta^{\ell(1-\sigma)})^{\frac{q-1}{\M}}
\pmod{\tilde\fq},
\end{equation}
where $q=|\mathcal{O}_K/\fq|$, $\eta$ is the top generator of the group $\ca{C}_L$ and $\hat\eta$ is defined in \eqref{tupe}.
\end{prop}

The proof of Proposition \ref{ChybejiciKongruence} is given further below. 
Assuming Proposition~\ref{ChybejiciKongruence} we may now finish the proof of Theorem~\ref{Semisp} by proving
the congruence relation (ii) in Definition \ref{semispe}. Using successively \eqref{UrceniBeta}, \eqref{Kongruence}, and \eqref{UrceniAlfa} we find that
\begin{align*}
\beta^{r(1-\sigma)^2\ell}=\hat\eta^{(-1)^s\ell(1-\sigma)\prod_{i=2}^{s-1}\Delta_{\dec_i}}&\equiv
\eta^{(-1)^s\frac{q-1}{\M}\ell(1-\sigma)\prod_{i=2}^{s-1}\Delta_{\dec_i}}\\&=
\alpha^{r\frac{q-1}{\M}\ell(1-\sigma)}
\pmod{\tilde\fq},
\end{align*}
where $r$ is the power of $p$ defined in \eqref{UrceniBeta}. Applying $\Delta_1$ to each side of the 
previous equality and using the facts that $\alpha^{N_1}=1$ (since $1\leq n$ and $\alpha^{N_n}=1$), that $(1-\sigma)\Delta_1=N_{1}-p^k$
and that $(\sigma-1)N_1=0$, we obtain
\begin{align}\label{trout}
\beta^{p^kr(1-\sigma)\ell}\equiv\alpha^{p^kr\frac{q-1}{\M}\ell}
\pmod{\tilde\fq}.
\end{align}
Because $\frac{q-1}{\M}\equiv1\equiv\ell\pmod\M$, it follows from \eqref{trout} that $\beta^{p^kr(1-\sigma)}$ and $\alpha^{p^kr}$
have the same image in $(\mathcal{O}_{L[\fq]}/{\tilde\fq})^\times/\M$. Moreover, since $r\mid p^{k(s-2)}$ it also follows that
$\beta^{1-\sigma}$ and $\alpha$ must have the same image in $(\mathcal{O}_{L[\fq]}/{\tilde\fq})^\times/(\M/p^{k(s-1)})$. 
We thus have shown that both $\varepsilon=\alpha$ and $\varepsilon_{\mk{q}}=\beta^{1-\sigma}$ satisfy the congruence relation (ii).  
This completes the proof of Theorem~\ref{Semisp}.
\end{proof}

\begin{proof}[Proof of Proposition~\ref{ChybejiciKongruence}]
The proof will follow essentially from an idea of Rubin, see \cite[Theorem 2.1]{RubinEuler}.
Let $\pi\in\mathcal{O}_K$ be such that $\pi\mathcal{O}_K=\fq^h$. Let $K_\M=K(\zeta_\M)$ where $\zeta_{\M}$ denotes a primitive
$\M$-th root of unity. Since 
$\mathcal{O}_K^\times=\mu_K$, $p\nmid|\mu_K|$ and $K_{\M}$ contains a primitive $p$-th root of unity, 
the field $M=K_\M(\pi^{\frac1p})$ does not depend on
the chosen generator $\pi$ of $\fq^h$ and on the chosen $p$-th root of $\pi$. One may also check that $M/K$ is a Galois extension.
Furthermore, we claim that $\pi$ cannot be a $p$-th power in $K_m$. Indeed, if it were the case then, since $p\nmid h$, this would imply that 
the ramification index of $\fq$ in $K_m/K$ would be divisible by $p$;
but this is impossible since $K_m/K$ ramifies only at primes above $p$.
Since $\pi$ is not a $p$-th power in $K_m$ it follows that
$M/K_\M$ is a cyclic extension of degree $p$. 

In order to finish the proof of Proposition \ref{ChybejiciKongruence} we need the following technical lemma:

\begin{lem}\label{PomocnyPrvoideal}
Let $\fq$ be as in Proposition \ref{ChybejiciKongruence} and recall that $\sigma$ is the unique ge\-ne\-ra\-tor of $\Gal(L[\mk{q}]/K[\mk{q}])$
which restricts to the initial generator of $\Gal(L/K)$ (which was also denoted by $\sigma$). 
Then there exists a prime $\fl$ of $K$ of absolute degree 1 satisfying the following three conditions:
\begin{enumerate}[(i)]
\item If we  let $\ell=|\mathcal{O}_K/\fl|$, then $\ell\equiv1\pmod \M$ and $\ell$ is unramified in $K/\Q$.
\item The prime $\fl$ is unramified in $L[\fq]$ and the Artin symbol $\bigl(\frac{L[\fq]/K}{\fl}\bigr)=\sigma^{-1}$.
\item The prime $\fq$ is inert in $K[\fl]/K$ (note that this is equivalent to say that $\fq$ is
unramified in $K[\fl]$ and that $\bigl\langle\bigl(\frac{K[\fl]/K}{\fq}\bigr)\bigr\rangle=\Gal(K[\fl]/K)$).
\end{enumerate}
\end{lem}
\noindent Recall here that the fields $K[\mk{q}],K[\fl]$ and $L[\mk{q}]$ were introduced in Definition \ref{neat4}.
Note that since $\fq\in\ca{Q}_m$ and $\sigma$ acts as the identity on $K[\fq]$, it follows from the above condition (ii) 
that $\fl$ splits completely in $K[\fq]/K$ into $\M$ distinct primes which stay inert in $L[\fq]/K[\fq]$. Moreover,
the fields $L[\fq]$ and $K[\fl]$ are linearly disjoint over $K$ since $\fl$ is unramified in $L[\fq]$ and
$\fl$ is totally ramified in $K[\fl]$.

Before finishing the proof of Proposition~\ref{ChybejiciKongruence} we find it more convenient to prove Lemma~\ref{PomocnyPrvoideal} first
and then finish the proof of Proposition~\ref{ChybejiciKongruence} afterwards. 
\begin{proof}[Proof of Lemma~\ref{PomocnyPrvoideal}] As the maximal abelian subextension of $M/K$ is $K_\M/K$ and $L[\fq]/K$ is abelian, we have 
$L[\fq]\cap M=L[\fq]\cap K_\M$. Since $L[\fq]/L$ is totally ramified at each prime above $\fq$ and $\fq$ is unramified in $K_\M/K$, 
we have that $L[\fq]\cap K_\M=L\cap K_\M$. As $p$ is unramified in $L/\Q$ and each prime above $p$ is totally ramified in $K_\M/K$, 
we also have that $L\cap K_\M=K$, and therefore, $L[\fq]\cap M=K$. Now, since $L[\mk{q}]$ and $M$ were shown to be linearly disjoint over $K$, 
there exists a $\tau\in\Gal((L[\fq]\cdot M)/K)$ which restricts to 
$\sigma^{-1}\in\Gal(L[\fq]/K)$ and to a generator of $\Gal(M/K_\M)\subseteq\Gal(M/K)$. 

By the \v Cebotarev's Density Theorem, there are infinitely many primes of $K$ of absolute degree 1 whose Artin symbol is the 
conjugacy class of $\tau$. We can choose among them a prime $\fl$ not dividing $6q\cdot q_1\dots q_s$ (here $q=|\ca{O}_K/\fq|$) such that $\ell=|\mathcal{O}_K/\fl|$
is unramified in $K/\Q$. Since $\tau$ acts as the identity on $K_\M$, it follows that $\ell$ splits completely in $\Q(\zeta_\M)/\Q$.
It is now clear that the first two conditions of the lemma are satisfied. 

It remains to prove the third condition. 
Let $\mathfrak{L}$ be a prime of $K_\M$ above $\fl$. Since $\fl$ splits completely in $K_\M/K$
it follows that $\mathcal{O}_{K_\M}/\mathfrak{L}\cong \mathcal{O}_K/\fl$. Moreover, because $\langle\tau\big|_{M}\rangle=\Gal(M/K_m)\cong \Z/p\Z$,
$\mathfrak{L}$ must be inert in $M/K_\M$. From these observations, it follows that the element $\pi$ cannot be a $p$-th power in 
$(\mathcal{O}_K/\fl)^\times$.

Recall that from Artin's Reciprocity Theorem and the fact that $p\nmid|\mu_K|$ we have 
$(\mathcal{O}_K/\fl)^\times/\M\cong\Gal(K[\fl]/K)$ (see \eqref{Artin}). Since $\pi$ was shown to be a non $p$-th power in 
$(\mathcal{O}_K/\fl)^\times$, 
it follows that $\bigl(\tfrac{K[\fl]/K}{\pi\mathcal{O}_K}\bigr)=\bigl(\tfrac{K[\fl]/K}{\fq}\bigr)^h$ is not a $p$-th power in $\Gal(K[\fl]/K)$.
Finally, since $\Gal(K[\fl]/K)$ is a cyclic group of order $m$ (a power of $p$), it follows that $\bigl(\tfrac{K[\fl]/K}{\fq}\bigr)$
must generate $\Gal(K[\mk{l}]/K)$, i.e., $\mk{q}$ is inert in $K[\mk{l}]/K$. This concludes the proof of Lemma \ref{PomocnyPrvoideal}.
\end{proof}

We may now finish the proof of Proposition~\ref{ChybejiciKongruence}. Recall that $\fq$ is a fixed prime in $\ca{Q}_m$.
Let $\mk{l}$ be a prime which satisfies the three conditions in Lemma~\ref{PomocnyPrvoideal}.
As in the proof of Theorem \ref{Semisp}, we let $\wp_{s+1}=\fq$, $F_{s+1}=K[\fq]$ and $I'=\{1,\dots,s+1\}$.
We introduce two auxiliary elliptic units:
\begin{align*}
\eta_\fl&=
\N_{K(\fl\mathfrak{m}_I)/L[\fl]}(\varphi_{\fl\mathfrak{m}_I})^{w_K},
\\
\hat\eta_\fl&=
\N_{K(\fl\mathfrak{m}_{I'})/L[\fq\fl]}(\varphi_{\fl\mathfrak{m}_{I'}})^{w_K},
\end{align*}
where $L[\fq\fl]$ means the compositum of $L[\fl]$ and $L[\fq]$
(for the definition of $\varphi_{\fl\mathfrak{m}_I}$ and $\varphi_{\fl\mathfrak{m}_{I'}}$ see \cite[Definition 2 on page 5]{H}). 
Since $\fl\nmid 6$, we have for any $\zeta\in\mu_K-\{1\}$ that $\zeta\not\equiv1\pmod\fl$.
Combining the previous observation with the norm relation \eqref{Relace},
and the fact that $\bigl(\frac{L[\fq]/K}{\fl}\bigr)=\sigma^{-1}$, we may deduce that
\begin{align}
\N_{L[\fq\fl]/L[\fl]}(\hat\eta_\fl)&=\eta_\fl^{q(1-\Fr(\fq)^{-1})},\label{Norma_ql_l}
\\
\N_{L[\fq\fl]/L[\fq]}(\hat\eta_\fl)&=\hat\eta^{\ell(1-\Fr(\fl)^{-1})}=\hat\eta^{\ell(1-\sigma)},\label{Norma_ql_q}
\\
\N_{L[\fl]/L}(\eta_\fl)&=\eta^{\ell(1-\Fr(\fl)^{-1})}=\eta^{\ell(1-\sigma)},\label{Norma_l}
\end{align}
where $q=|\ca{O}_K/\fq|$, $\ell=|\ca{O}_K/\fl|$, $\Fr(\fq)=\bigl(\frac{L[\mk{l}]/K}{\mk{q}}\bigr)$ and
$\Fr(\fl)=\bigl(\frac{L[\mk{q}]/K}{\mk{l}}\bigr)$. In order to compare the different units $\hat\eta_\fl, \eta_\fl,\hat\eta$ and $\eta$,
we shall work in $\mathcal{O}_{L[\fq\fl]}$ modulo the product of all the primes of $L[\fq\fl]$ above $\fq$, which we denote by $\hat\fq$.
Since $\fq\in\ca{Q}_m$, $\mk{q}$ splits completely in $L/K$, 
and by the third condition of Lemma~\ref{PomocnyPrvoideal}, the primes of $L$ above $\fq$ are inert in $L[\fl]/L$. 
Therefore, each prime of $L[\fq]$ above $\fq$ is inert in $L[\fq\fl]/L[\fq]$, 
and so $\hat\fq=\tilde\fq\mathcal{O}_{L[\fq\fl]}$, where as before $\tilde{\mk{q}}$ corresponds to the
product of all primes of $L[\mk{q}]$ above $\mk{q}$. We therefore have the following isomorphisms of rings:
\begin{align*}
\mathcal{O}_{L[\fq]}/\tilde\fq&\cong\mathcal{O}_L/\fq\mathcal{O}_L\cong
(\mathbb{F}_{q})^{p^k},
\\
\mathcal{O}_{L[\fq\fl]}/\tilde\fq\mathcal{O}_{L[\fq\fl]}&\cong
\mathcal{O}_{L[\fl]}/\fq\mathcal{O}_{L[\fl]}\cong(\mathbb{F}_{q^m})^{p^k}.
\end{align*}
Since $L[\fq]$ and $L[\fl]$ are linearly disjoint over $L$, it makes sense to
extend $\Fr(\fq)\in\Gal(L[\fl]/K)$ to $L[\fq\fl]$ in such a 
way that $\Fr(\fq)$ is the identity on $L[\fq]$ and we still denote this extension by $\Fr(\fq)$. In particular, 
$\Fr(\fq)$ generates $\Gal(L[\fq\fl]/L[\fq])$.

It follows from the discussion above that $\Fr(\fq)$ acts as raising to the $q$-th power on $\mathcal{O}_{L[\fq\fl]}/\tilde\fq\mathcal{O}_{L[\fq\fl]}$,
and that $\Gal(L[\fq\fl]/L[\fl])$ (the inertia group at $\fq$) acts trivially on $\mathcal{O}_{L[\fq\fl]}/\tilde\fq\mathcal{O}_{L[\fq\fl]}$. 
From these two observations, it follows that the norms $\Norm_{L[\fq\fl]/L[\fl]}$ and $\Norm_{L[\fq\fl]/L[\fq]}$ act on the ring
$\mathcal{O}_{L[\fq\fl]}/\tilde\fq\mathcal{O}_{L[\fq\fl]}$ as raising to the $\M$-th power and  
as raising to the $\bigr(\sum_{i=0}^{\M-1}q^i\bigl)$-th power, respectively. 
Since $q\equiv 1\pmod\M$, there exists a positive integer $r$ such that $\sum_{i=0}^{\M-1}q^i=\M r$.

Combining \eqref{Norma_ql_q}, \eqref{Norma_ql_l}, and \eqref{Norma_l} we find that
\begin{align}\label{dent}
\hat\eta^{\ell(1-\sigma)}&\equiv
\hat\eta_\fl^{\M r}\equiv
\eta_\fl^{qr(1-\Fr(\fq)^{-1})}\equiv
\eta_\fl^{r(q-1)}\equiv
(\eta_{\fl}^{mr})^{\frac{q-1}{\M}}\nonumber\\&\equiv
\eta^{\ell(1-\sigma)\frac{q-1}\M}
\pmod{\tilde\fq\mathcal{O}_{L[\fq\fl]}}.
\end{align}
Finally, since the natural map $\ca{O}_{L[\fq]}/\tilde\fq\rightarrow \ca{O}_{L[\fq\fl]}/\tilde\fq\ca{O}_{L[\fq\fl]}$ is injective,
it follows from \eqref{dent} that $\hat\eta^{\ell(1-\sigma)}\equiv\eta^{\ell(1-\sigma)\frac{q-1}\M}\pmod{\tilde{\fq}}$.
This completes the proof of Proposition~\ref{ChybejiciKongruence}.
\end{proof}

\section{Annihilating the ideal class group}

For this section we keep the same notation and assumptions as in the previous sections. In particular, 
$\Gal(L/K)=\Gamma=\langle\sigma\rangle\cong\Z/p^k\Z$ and the extended group of 
elliptic units $\overline\GrEllUn_L$ is defined as the $\Z[\Gamma]$-submodule of ${\mathcal O}^\times_{L}$ generated by $\mu_K$ and by the units $\alpha_1,\dots,\alpha_k$, 
see Definition~\ref{DefiniceAlfa_i}. 

For each $j\in\{1,\dots,s\}$ recall that $n_j$ (a power of $p$) was defined as the index of the decomposition group of $\mathfrak{P}_j$ (a prime of $L$ above
$\wp_j$) in $\Gamma$ (see Section \ref{bit}) and that 
$\dec_1\le\dec_2\le\ldots\le\dec_s$ (see \eqref{Serazeni}).
For each $i\in\{1,\dots,k\}$ we define 
\begin{align}\label{mu}
\mu_i=\dec_{\max M_i},
\end{align} 
where $M_i\subseteq\{1,\ldots,s\}$ 
is the set defined in \eqref{DefiniceM}. In particular, $\mu_i$ is always a power of $p$ (possibly trivial).
Since $M_{i}\subseteq M_{i+1}$, we always have that $\mu_i\leq \mu_{i+1}$. Let us call an index $i\in \{1,\ldots, k-1\}$ a {\it jump}
if $\mu_{i}<\mu_{i+1}$.
Furthermore, we declare  the indices $0$ and $k$ to be jumps and we set $\mu_0=0$. 
Using the notion of jumps one can write down a $\Z$-basis of $\overline\GrEllUn_L$ using only conjugates of the
generators $\alpha_1,\dots,\alpha_k$ whose indices correspond to jumps.

\begin{lem}\label{DruhaBazeNoveGrupy}
Let $0=s_0 < s_1 < \ldots < s_\kappa = k$ be the ordered sequence of all the jumps. Note that $\kappa\geq 1$.
Then the set $\bigcup_{t=1}^\kappa\{\alpha_{s_t}^{\sigma^i};\,0\le i<p^{s_t}-p^{s_{t-1}}\}$ is a $\Z$-basis of $\overline\GrEllUn_L$.
\end{lem}
\begin{proof}
This is proved along similar lines to those in \cite[Lemma 5.1]{AnnIII}. Let us just point out the two main ideas.
For each $1<i\le k$ one can show that
\begin{equation}\label{NormaAlfy}
\N_{L_i/L_{i-1}}(\alpha_i)\in\langle\alpha_{i-1}\rangle_{\Z[\Gamma]},
\end{equation}
and furthermore, for each $0<u<v\le k$ such that $\mu_u=\mu_v$, one may prove the stronger result that
\begin{equation}\label{NormaGeneratoru}
\langle\N_{L_v/L_u}(\alpha_v)\rangle_{\Z[\Gamma]}=
\langle\alpha_u\rangle_{\Z[\Gamma]}.
\end{equation}
This concludes the sketch of the proof.
\end{proof}

From the explicit $\Z$-basis for $\overline\GrEllUn_L$ which appears in Lemma \ref{DruhaBazeNoveGrupy} we easily deduce the following:

\begin{lem}\label{Jednoznacnost}
Let $r$ be the highest jump less than $k$, i.e., $\mu_r<\mu_{r+1}=\dec_s$ where $n_s$ is defined in \eqref{Serazeni}.
Let us assume that $\rho\in\Z[\Gamma]$ is such that $\alpha_k^\rho\in\overline\GrEllUn_{L_r}$. Then 
\begin{align}\label{pouce}
(1-\sigma^{p^r})\rho=0.
\end{align}
\end{lem}
\begin{proof}
There is a unique polynomial $f\in\Z[x]$ with $\deg f<p^k$, such that $\rho=f(\sigma)$. Let $\phi=x^{p^k-p^r}+\dots+x^{2p^r}+x^{p^r}+1$.
From the euclidean division of $f$ by $\phi$ there exist polynomials $Q,g\in\Z[x]$ such that $f=\phi\cdot Q+g$ where $\deg g<p^k-p^r$.
By assumption, we have $\alpha_k^\rho\in\overline\GrEllUn_{L_r}$, and from \eqref{NormaGeneratoru} we know that 
$\alpha_k^{\phi(\sigma)}=\N_{L_k/L_r}(\alpha_k)\in\overline\GrEllUn_{L_r}$; combining these two relations we obtain
that $\alpha_k^{g(\sigma)}=\frac{\alpha_k^{f(\sigma)}}{\alpha_k^{\phi(\sigma)Q(\sigma)}}\in\overline\GrEllUn_{L_r}$. 
Since $\{\alpha_k$, $\alpha_k^\sigma, \dots$, $\alpha_k^{\sigma^{p^k-p^r-1}}\}$ is a 
part of the $\Z$-basis given in
Lemma~\ref{DruhaBazeNoveGrupy}, and the rest of this $\Z$-basis, namely $\bigcup_{t=1}^{\kappa-1}\{\alpha_{s_t}^{\sigma^i};\,0\le i<p^{s_t}-p^{s_{t-1}}\}$,
is also a $\Z$-basis of $\overline\GrEllUn_{L_r}$ (using again Lemma \ref{DruhaBazeNoveGrupy}); we deduce that $g=0$. In particular,
$\rho=(1+\sigma^{p^r}+\sigma^{2p^r}+\dots+\sigma^{p^k-p^r})\rho'$ 
for some $\rho'\in\Z[\Gamma]$, and thus \eqref{pouce} follows.
\end{proof}

From Theorem~\ref{Indexy}, we know that $\mathcal{O}_L^\times/\overline\GrEllUn_L$ is a finite $\Z[\Gamma]$-module.
Let $(\mathcal{O}_L^\times/\overline\GrEllUn_L)_p$ and $\Cl_p$ denote the $p$-Sylow subgroups of the corresponding $\Z[\Gamma]$-modules.
The aim of this section is to construct annihilators of $\Cl_p$ by means of annihilators
of $(\mathcal{O}_L^\times/\overline\GrEllUn_L)_p$. 
To do this we appeal to the following key theorem which allows one to produce annihilators of $\Cl_p$ from certain units of $L$.
This theorem should be viewed as a modification of a similar result
obtained first by Thaine (see Proposition 6 of \cite{Tha}) and then generalized by Rubin (see Theorem 5.1 of \cite{Rubin}). 

\begin{thm}\label{Theorem12}
Let $\M$ be a power of $p$ divisible by $p^{ks}$.
Assume that $\varepsilon\in \mathcal{O}_L$ is 
$\M$-semispecial, suppose that $V\subseteq L^\times/\M$ is 
a finitely generated $\Z[\Gamma ]$-submodule, and that the class containing $\varepsilon$ belongs to $V$. Let 
$\zobr :\,V\to \Z/\M[\Gamma ]$ be a $\Z[\Gamma ]$-linear map
such that $\zobr(V\cap K^\times)=0$, where $V\cap K^\times$ is taken to mean
$V\cap( K^\times{L^\times}^\M/{L^\times}^\M)$. Then 
$\zobr (\varepsilon)$ annihilates $\Cl_p/(\M/p^{k(s-1)})$.
\end{thm}
\begin{proof}
This can be proved along similar lines as those in \cite[Theorem 12]{AnnI}. In order to guide the reader to make the necessary 
modifications needed for the proof we chose to state below Theorem~\ref{Theorem17} (the needed version of 
\cite[Theorem 17]{AnnI} which has its origin in \cite[Theorem 5.5]{Rubin}). This concludes our rough sketch of the proof. 
\end{proof}

\begin{thm}\label{Theorem17}
Fix a $p$-power $\M$, suppose that $V\subseteq L^\times/\M$ is 
a finitely generated $\Z_p[\Gamma]$-submodule. Without loss of generality we may assume that we have chosen a set of generators of $V$ which belongs to $\mathcal{O}_L$. 
Let us suppose that we are given a $\Z_p[\Gamma]$-linear map $\zobr :\,V\to \Z/\M[\Gamma]$ which is such that
$\zobr (V\cap K^\times)=0$. Then, for any $\mathfrak c\in\Cl_p$, there exist infinitely many unramified primes $\fQ$ in $L$ of
absolute degree $1$ satisfying the following conditions:

Let $\fq$ be the prime ideal of $K$ below $\fQ$ and
$q$ be the rational prime number below $\mathfrak q$.
\begin{enumerate}
\item[(i)] $[\fQ]=\mathfrak c$, where $[\fQ]$ is the projection of the ideal class of  $\fQ$ into $\Cl_p$;
\item[(ii)] $q\equiv1+\M\pmod{\M^2}$;
\item[(iii)] for each $j=1,\dots,s$, the class of $\pi_j$ is an $\M$-th power in $(\mathcal{O}_K/\fq)^\times$;
\item[(iv)] the support of any of the chosen generators of $V$ does not contain any prime of $L$ above $\mathfrak q$, and
there is a $\Z_p[\Gamma]$-linear map $\varphi :\,(\mathcal{O}_L/\fq)^\times/\M\to\Z/\M[\Gamma]$ such that the following diagram
$$
\xymatrix{ V \ar[r]^\zobr \ar[d]^\psi&\Z/\M[\Gamma]\\
(\mathcal{O}_L /\fq)^\times/\M \ar@{-->}[ur]_\varphi&}
$$
commutes, where $\psi$ corresponds to the reduction map.
\end{enumerate}
\end{thm}
\begin{proof}
This can be proved in the same way as \cite[Theorem 17]{AnnI}.
\end{proof}

We may finally present the main result of this paper.
\begin{thm}\label{Theorem4}
Let $r$ be the highest jump less than $k$, i.e., $\mu_r<\mu_{r+1}=\dec_s$.
If $\varkappa\in\ann{(\mathcal{O}_L^\times/\overline\GrEllUn_L)_p}$,
then $(1-\sigma^{p^r})\varkappa$ annihilates $\Cl_p$. In other words, we have
$$
\ann{(\mathcal{O}_L^\times/\overline\GrEllUn_L)_p}\subseteq
\ann{(1-\sigma^{p^r})\Cl_p}.
$$
The  number $r$ can be characterized as follows: 
$p^{k-r}=\max\{t_j;\,j\in J\}$, where $J=\{j\in\{1,\dots,s\};\,\dec_j=\dec_s\}$.
\end{thm}
\begin{proof}
Fix a $p$-power $\M$ which is large enough so that $\M\nmid p^{ks}h_L$ and let
$$
\varkappa\in\ann{(\mathcal{O}_L^\times/\overline\GrEllUn_L)_p}
$$ 
be a fixed annihilator.
We shall first construct a $\Z[\Gamma ]$-linear map $\zobr':\,\mathcal{O}_L^\times\to\Z[\Gamma]$, that will only depend on the annihilator $\varkappa$, 
and then consider the induced
map $z:\,\mathcal{O}_L^\times/m\to\Z/m[\Gamma]$. Let $f$ be the greatest divisor of the index $[\mathcal{O}_L^\times:\overline\GrEllUn_L]$ which is not divisible by $p$.
Then
$$
f\varkappa\in\ann{\mathcal{O}_L^\times/\overline\GrEllUn_L},
$$ 
and thus, for any unit $\varepsilon\in\mathcal{O}_L^\times$, we have $\varepsilon^{f\varkappa}\in\overline\GrEllUn_L$. From 
Lemma~\ref{DruhaBazeNoveGrupy}, there is $\rho\in\Z[\Gamma]$ and $\delta\in\overline\GrEllUn_{L_r}$ such that $\varepsilon^{f\varkappa}=\delta\alpha_k^\rho$. 
We define $\zobr'(\varepsilon)=(1-\sigma^{p^r})\rho$. Let us check that the the map $z'$ is well-defined. If
$\varepsilon^{f\varkappa}=\delta'\alpha_k^{\rho'}$ for some $\rho'\in\Z[\Gamma]$ and $\delta'\in\overline\GrEllUn_{L_r}$, then 
$\alpha_k^{\rho-\rho'}=\delta'\delta^{-1}\in\overline\GrEllUn_{L_r}$; applying Lemma~\ref{Jednoznacnost} we find that
$(1-\sigma^{p^r})(\rho-\rho')=0$, and so $z'$ is well-defined. It follows directly from the definition of $z'$ that $\zobr'(\alpha_k)=(1-\sigma^{p^r})f\varkappa$ and 
that $\zobr'(\varepsilon)=0$ if $\varepsilon\in\mathcal{O}_L^\times\cap K^\times=\mu_K$. 

Let $V=\mathcal{O}_L^\times/\M$. We want to apply Theorem~\ref{Theorem12} to the $\Z_p[\Gamma]$-linear map 
$\zobr :\,V\to \Z/\M[\Gamma ]$ determined by the map $\zobr'$. Now, from Theorem~\ref{Semisp}, we know that $\alpha_k\in\ca{O}_L^{\times}$ is $\M$-semispecial 
and therefore, from Theorem \ref{Theorem12},
we obtain that $z(\alpha_k)=(1-\sigma^{p^r})f\varkappa$ annihilates $\Cl_p/(\M/p^{k(s-1)})$. 
Finally, since $p\nmid f$ and $\M\nmid p^{ks}h_L$, it follows that $\Cl_p/(\M/p^{k(s-1)})=\Cl_p$, and therefore $(1-\sigma^{p^r})\varkappa$ annihilates
$\Cl_p$.

It remains to prove the last equality in Theorem \ref{Theorem4} which gives a characterization of the index $r$. Recall that for each index $i\in\{1,\ldots,k\}$,
$M_i=\{j\in\{1,\ldots,s\}:t_j>p^{k-i}\}$ by \eqref{DefiniceM} and that $\mu_i=n_{\max M_i}$ by \eqref{mu}. It follows from the definitions of $J$ and $\mu_i$  that
\begin{align}\label{equi}
\mu_i<n_s\Longleftrightarrow M_i\cap J=\emptyset.
\end{align} 
In particular, if we set $i=r$ in \eqref{equi} we find that
$M_r\cap J=\emptyset$ and therefore, for each $j\in J$ we must have the inequality (a) $t_j\leq p^{k-r}$.
Let us show that the reverse inequality holds true for at least one index. 
Since $\mu_{r+1}=n_s$ it follows from \eqref{equi} that $M_{r+1}\cap J\neq\emptyset$. Hence there must exist
at least one index
$j_0\in M_{r+1}\cap J$; and by definition
of $M_{r+1}$, we must have that (b) $t_{j_0}>p^{k-(r+1)}$. Finally, combining the inequalities (a) and (b) together we find that
$t_{j_0}=p^{k-r}$ and thus $p^{k-r}=t_{j_0}=\max\{t_j;\,j\in J\}$.
\end{proof}


\end{document}